\documentclass[11pt,reqno]{amsart}
\usepackage{CJK}
\usepackage{hyperref}
\usepackage{cite}
\usepackage{amsmath}
\usepackage{mathrsfs}

\makeatletter
\@namedef{subjclassname@2020}{%
	\textup{2020} Mathematics Subject Classification}
\makeatother

\numberwithin{equation}{section}

\newtheorem{theorem}{Theorem}[section]
\newtheorem{lemma}[theorem]{Lemma}
\newtheorem{definition}[theorem]{Definition}

\newtheorem{proposition}[theorem]{Proposition}

\allowdisplaybreaks

\begin{document}
\title[\hfil Regularity for fully nonlinear nonlocal elliptic equations] {Regularity for a class of degenerate fully nonlinear nonlocal elliptic equations}
\author[Y. Fang, V.D. R\u{a}dulescu, C. Zhang]{Yuzhou Fang, Vicen\c{t}iu D. R\u{a}dulescu$^*$ and Chao Zhang}

\thanks{$^*$Corresponding author.}

\address{Yuzhou Fang\hfill\break School of Mathematics, Harbin Institute of Technology, Harbin 150001, China}
 \email{18b912036@hit.edu.cn}

\address{Vicen\c{t}iu D. R\u{a}dulescu   \hfill\break   Faculty of Applied Mathematics,
	AGH University of Science and Technology,  Krak\'{o}w 30-059, Poland \&   Department of Mathematics, University of Craiova, Craiova 200585, Romania}
\email{radulescu@inf.ucv.ro}

\address{Chao Zhang\hfill\break School of Mathematics and Institute for Advanced Study in Mathematics, Harbin Institute of Technology, Harbin 150001, China}
 \email{czhangmath@hit.edu.cn}


\begin{abstract}
We consider a wide class of fully nonlinear integro-differential equations that degenerate when the gradient of the solution vanishes. By using compactness and perturbation arguments, we give a complete characterization of the regularity of viscosity solutions according to different diffusion orders. More precisely, when the order of the fractional diffusion is sufficiently close to 2, we obtain H\"{o}lder continuity for the gradient of any viscosity solutions and further derive an improved gradient regularity estimate at the origin. For the order of the fractional diffusion in the interval $(1, 2)$, we prove that there is at least one solution of class $C^{1, \alpha}_{\rm loc}$. Additionally, if the order of the fractional diffusion is in the interval $(0,1]$, the local H\"{o}lder continuity of solutions is inferred.
\end{abstract}

\maketitle

\section{Introduction}
\label{sec0}

In this paper, we concentrate on the local behaviors of viscosity solutions to a quite general class of fully nonlinear nonlocal elliptic problems of the type
\begin{equation}
\label{main}
-\Phi(x,|Du|)\mathcal{I}_\sigma(u,x)=f(x)  \quad\text{in } B_1,
\end{equation}
where $B_1:=B_1(0)$ is a unit ball in the Euclidean space $\mathbb{R}^n$, $f\in C(B_1)\cap L^\infty(B_1)$, $\Phi: B_1\times[0,\infty)\rightarrow[0,\infty)$ is a continuous map possessing degeneracy as the gradient vanishes, and the nonlinear nonlocal operator $\mathcal{I}_\sigma$ is uniformly elliptic in the sense of Caffarelli and Silvestre \cite{CS09, CS11}, that is, there holds
$$
\inf_{I\in\mathcal{L}}Iv(x)\le \mathcal{I}_\sigma(u+v,x)-\mathcal{I}_\sigma (u,x)\le\sup_{I\in\mathcal{L}}Iv(x)
$$
for a family of linear operators $\mathcal{L}$.  Let $\sigma\in(0,2)$ and $0<\lambda\le\Lambda<\infty$, $\mathcal{K}$ be a family of symmetric kernels formed by the measurable functions $K:\mathbb{R}^N\setminus\{0\}\rightarrow\mathbb{R}^+$ having
$$
\lambda\frac{C_{N,\sigma}}{|x|^{N+\sigma}}\le K(x)\le \Lambda\frac{C_{N,\sigma}}{|x|^{N+\sigma}},
$$
where $C_{N,\sigma}>0$ is a normalizing constant. 
For $K\in \mathcal{K}$ and $u:\mathbb{R}^N\rightarrow\mathbb{R}$, define
$$
I_Ku(x)=\frac{1}{2}\mathrm{P.V.}\int_{\mathbb{R}^N}(u(x+y)+u(x-y)-2u(x))K(y)\,dy
$$
with the symbol P.V. meaning the Cauchy principal value. Observe that $I_Ku$ for each $K$ is well-defined if $u$ is $C^{1,1}$ at the point $x$ and fulfills adequate growth conditions at infinity, i.e.,
$$
\|u\|_{L^1_\sigma(\mathbb{R}^N)}:=\int_{\mathbb{R}^N}\frac{|u(y)|}{1+|y|^{N+\sigma}}\,dy<+\infty.
$$
At this time, we denote by $L^1_\sigma(\mathbb{R}^N)$ the set of such functions. For a two-parameter collection of kernels $\{K_{\alpha\beta}\}_{\alpha\beta}\subseteq\mathcal{K}$, the nonlinear operator $\mathcal{I}_\sigma$ can be given by
$$
\mathcal{I}_\sigma u(x)=\inf_\beta\sup_\alpha I_{K_{\alpha\beta}}u(x),
$$
which appears naturally in stochastic control problems when two or more players are involved in a competitive game, see \cite{Soner}.

In the past years, integro-differential operators have attracted increasing attention and have developed abundant theories. For the qualitative properties, Barles, Chasseigne and Imbert \cite{BCI08} considered a Dirichlet problem for elliptic integro-differential equations and provided general existence results of viscosity solutions by Perron's method. The comparison principles for second-order degenerate integro-differential equations were shown in \cite{BI08}, whereas a key ingredient was a nonlocal version of Jensen Ishii's lemma for solutions with arbitrary growth at infinity to such problems; see also \cite{JK06} for similar results on subquadratic solutions. As for the parabolic analogue, Ciomaga \cite{Cio12} verified the strong maximum principle by studying the horizontal and vertical propagation of maxima. More existence, uniqueness, and comparison results can be referred to e.g. \cite{AT96, Ama03, JK05, Pha98}.

When it comes to the quantitative properties of solutions, Bass and Levin \cite{BL08} first got, relying on probabilistic methods, Harnack inequalities for nonnegative harmonic functions concerning a class of integro-differential elliptic equations in the framework of potential theory. Then, such results were extended to the problems exhibiting more general structures in \cite{BK05, SZ04}. In \cite{Sil06}, Silvestre presented instead a purely analytical proof for H\"{o}lder continuity of harmonic functions concerning the integral operator with PDE techniques, where he also dealt with quite particular nonlinear nonlocal equations. It is worth mentioning that the previous quantitative estimates will blow up when the order of the equation goes to 2.

The first results being uniform in the degree were derived by \cite{CS09}, in which Caffarelli and Silvestre investigated general fully nonlinear integro-differential equations and demonstrated interior behaviors of solutions via a powerful Harnack approach introduced by Krylov-Safonov. To be precise, the authors justified H\"{o}lder continuity by establishing the nonlocal ABP estimate and Harnack inequality ahead of time, which leads further to $C^{1,\alpha}$-regularity under some extra assumptions on the integral kernels. Because these estimates remain uniform as the degree of operator tends to 2, they can be viewed as a natural extension of the regularity theory for PDEs. Subsequently, the same authors \cite{CS11} continued to generalize $C^{1,\alpha}$-regularity to the nonlocal equations that are not translation-invariant by utilizing compactness and perturbative argument. That is, the solutions to the equation studied are $C^{1,\alpha}$ regular, whenever this equation is uniformly close to another one with $C^{1,\alpha}$ solutions. More or less simultaneously, Barles et al. \cite{BCI11} concluded, using the Ishii-Lion's viscosity method, H\"{o}lder estimates for a large class of elliptic and parabolic integro-differential equations involving second and first-order terms; see \cite{BCCI12} concerning further Lipschitz regularity in a similar setting.

In particular, some specific cases of \eqref{main} have been explored up to now. For instance, if $\Phi(x,t)=t^\gamma$, dos Prazeres and Topp \cite{PT21} proved the local H\"{o}lder and Lipschitz estimates of viscosity solutions by following the ideas due to \cite{BCI11, BCCI12}, together with gradient H\"{o}lder continuity via an improvement of the flatness procedure for $\sigma$ close enough to 2. See also \cite{BKLT15} for the situation $\gamma<-\sigma$. Afterwards, these results were extended to the nonhomogeneous scenario, i.e., $\Phi(x,t)=t^p+a(x)t^q$, in \cite{APS24}. Especially, for any $1<\sigma<2$, the authors in \cite{APS24} discovered that there exists at least one viscosity solution being $C^{1,\alpha}$ regular under the degenerate setting $(0<p\le q)$, which is new even for $\Phi(x,t)=t^\gamma$ in \cite{PT21}. Recently, several aspects of nonlocal quasilinear equations with elliptic degeneracy could be found in \cite{APT23}, including existence, multiplicity and gradient regularity.

The investigation of model \eqref{main} is also motivated by physical contexts such as porous medium flow, dislocation dynamics and others in \cite{APS24}. On the other hand, Eq. \eqref{main} can be seen as a nonlocal counterpart to second-order fully nonlinear equations with the form
\begin{equation}
\label{z}
-\Phi(x,|Du|)F(D^2u)=f \quad \text{in } B_1,
\end{equation}
where the operator $F:\mathcal{S}^N\rightarrow\mathbb{R}$ is uniformly elliptic in the sense that
$$
\lambda\mathrm{Tr}(B)\le F(A+B)-F(A)\le\Lambda\mathrm{Tr}(B)
$$
for all $A,B\in\mathcal{S}^N$, $B\ge0$. Here $\mathcal{S}^N$ is the set of symmetric matrices. Let us point out that Eq. \eqref{z} is a nonvariational analogue of the well-known double phase equation, for instance, \cite{CM15, DeFM, JMAA, FRZ24, HO22, Mar89}. The interior and global regularity results related to this kind of equation were inferred by \cite{IFB, MZ}; see \cite{BBO24} for second derivative $L^\delta$-estimates. In addition, some special types of \eqref{z} have been extensively investigated, such as $\Phi(x,t)=t^{p(x)}+a(x)t^{q(x)}$ with $-1<p(x)\le q(x)$ and $a(x)\ge0$ in \cite{FRZ21, JLMS}. We refer the readers to \cite{Caf89, IS13, DeF21, SR20, Nas24, BPRT20} and references therein for more studies.

Inspired by the work mentioned above, in this paper we consider the fully nonlinear nonlocal elliptic equation \eqref{main} with a quite general degeneracy of the type $\Phi(\cdot, |Du|)$, which covers all the previous models. Our equation can not only encompass the $p$-growth and double-phase growth degeneracy but also include the variable exponent, $\log$-type and Orlicz double-phase growth cases. We aim at seeking some appropriate structure conditions on \eqref{main} to establish the interior $C^{0,\alpha}$, $C^{1,\alpha}$ regularity in a universal way. To begin with, by following the approximation idea described in \cite{IS13}, we give gradient H\"{o}lder continuity result under the case that $\sigma$ is sufficiently close to 2. Since the nonlocal operator $\mathcal{I}_\sigma$ at this time can approach a uniformly elliptic local operator $F$, the regularity of the solutions to \eqref{main} will be transferred from that of the $F$-harmonic functions in a suitable manner.

\begin{theorem}
\label{thm1}
Let $u\in C(\overline{B}_1)$ be a viscosity solution of \eqref{main} and the conditions $(A_1)$--$(A_5)$ (in Section \ref{sec1}) be in force. Then there is a $\sigma_0\in (1,2)$, close enough to 2, such that $u$ is locally of class $C^{1,\alpha}(B_1)$ with
$$
\|u\|_{C^{1,\alpha}(B_{1/2})}\le C\left(\|u\|_{L^\infty(B_{1})}+\|u\|_{L^1_\sigma(\mathbb{R}^N)}+\|f\|^\frac{\sigma-1}{1+s_1}_{L^\infty(B_{1})}\right),
$$
provided $\sigma\in(\sigma_0,2)$, where $0<\alpha<\min\left\{\overline{\alpha},\frac{\sigma-1}{1+s_2}\right\}$ and the positive constant $C$ depends on $N,\lambda,\Lambda,s_1,M,\overline{M},\alpha$. Here $\overline{\alpha}$ is the index related to the optimal regularity for an $F$-harmonic function.
\end{theorem}

The constant $C$ in Theorem \ref{thm1} is uniform in $\sigma$, that is, it will not blow up as $\sigma\rightarrow2$. The result stated in the theorem above is also sharp owing to an example given by \cite{PT21} in the spirit of \cite{ART15}. Nonetheless, saying simply that solutions are $C^{1,\alpha}$ regular does not tell the whole story. For instance, if $\Phi(x,t)=t^p$ and $\mathcal{I}_\sigma$ takes the fractional Laplacian, for each $\beta\in(0,\sigma-1)$, a function $u(x)=|x|^{1+\beta}$ admits
$$
|Du|^p(-\Delta)^\frac{\sigma}{2}u(x)=C|x|^{(1+\beta)(1+p)-(\sigma+p)},
$$
see \cite{DelQ23}. Letting $\theta=(1+\beta)(1+p)-(\sigma+p)$, we can find that $\theta$ is smaller than or equal to the degenerate rate $p$, and that $u$ is $C^{1,\frac{\sigma-1+\theta}{1+p}}$ at the origin (a critical point of $u$); see \cite{Nas24} for the local version. This suggests that the $C^{1,\alpha}$ regularity of $u$ can be surpassed, at least at some meaningful point, if H\"{o}lder exponent of the right-hand side is larger than the growth rate of gradient term in the equation. Therefore, we can further deduce an improved gradient estimate for the scenario $\Phi(x,t)=t^p$, which is new as far as we know. It should be noted that any universal constant mentioned in this paper  means that it might depend only on the parameters related to Eq. \eqref{main} itself.

\begin{theorem}
\label{thm2}
Let the assumptions $(A_1),(A_2),(A_5)$ be in force. Suppose that $u\in C(\overline{B}_1)$ is a viscosity solution of $|Du|^p\mathcal{I}_\sigma(u,x)=f(x)$ with the source term $|f(x)|\le\mathfrak{m}|x|^\theta$ for $\theta\in(0,1)$. Then one can find a $1<\sigma_0<2$ (sufficiently close to 2) such that $u$ is of class $C^{1,\min\left\{\overline{\alpha}^-,\frac{\sigma-1+\theta}{1+p}\right\}}$ at the origin with the estimate
$$
|u(x)-u(0)-Du(0)\cdot x|\le C|x|^{1+\alpha},  \quad x\in B_{\frac{1}{8}}
$$
for any
$$
\alpha\in (0,\overline{\alpha})\cap\left(0,\frac{\sigma-1+\theta}{1+p}\right],
$$
provided $\sigma_0<\sigma<2$. Here $\overline{\alpha}$ is an exponent corresponding to the optimal regularity of an $F$-harmonic function, and $C>0$ depends upon $\alpha$ and universal parameters.
\end{theorem}

In the general case $1<\sigma<2$, we could not apply directly the method used to prove Theorem \ref{thm1}, because $\mathcal{I}_\sigma$ may not approximate a uniformly elliptic local operator. However, similar to \cite[Theorem 1.1]{APS24}, we still verify that there exists at least one solution $u$ to \eqref{main} satisfying $u\in C^{1,\alpha}_{\rm loc}(B_1)$ for some $\alpha\in(0,1)$, which is stated as below.

\begin{theorem}
\label{thm3}
If $\sigma\in(1,2)$ and the assumptions $(A_1)$--$(A_4)$ hold true, we can find such a viscosity solution $u\in C(\overline{B}_1)$ of \eqref{main} that $u$ is $C^{1,\alpha}_{\rm loc}(B_1)$ regular and
$$
\|u\|_{C^{1,\alpha}(B_{1/2})}\le C\left(\|u\|_{L^\infty(B_{1})}+\|u\|_{L^1_\sigma(\mathbb{R}^N)}+\|f\|^\frac{\sigma-1}{1+s_1}_{L^\infty(B_{1})}\right),
$$
where $0<\alpha<\min\left\{\hat{\alpha}, \frac{\sigma-1}{1+s_2}\right\}$ and $C>0$ depends on $\alpha$ and universal constant. 
Here $\hat{\alpha}$ is the index related to the regularity of $\mathcal{I}_\sigma$-harmonic functions.
\end{theorem}

Finally, let us present lower-order regularity shown by the Ishii-Lion's method in \cite{BCCI12, BCI11} under the case $0<\sigma\le1$.

\begin{theorem}
\label{thm4}
Suppose that the conditions $(A_1)$--$(A_4)$ are satisfied. Let $u\in C(\overline{B}_1)$ be a viscosity solution to 
\eqref{main}. Then,
\begin{itemize}
  \item[(1)] for $\sigma\in(0,1)$, $u\in C^{0,\sigma}_{\rm loc}(B_1)$ and
  $$
  \|u\|_{C^{0,\sigma}(B_{1/2})}\le C\left(\|u\|_{L^\infty(B_1)}+\|u\|_{L^1_\sigma(\mathbb{R}^N)}+\|f\|_{L^\infty(B_1)}\right)
  $$
  with $C$ depending on $\sigma,N,\lambda,\Lambda,M,\overline{M},s_1$;

    \smallskip

  \item[(2)] for $\sigma=1$, $u\in C^{0,\alpha}_{\rm loc}(B_1)$ with any $\alpha\in(0,1)$, and
  $$
  \|u\|_{C^{0,\alpha}(B_{1/2})}\le C \left(\|u\|_{L^\infty(B_1)}+\|u\|_{L^1_\sigma(\mathbb{R}^N)}+\|f\|_{L^\infty(B_1)}\right)
  $$
  with $C$ depending on $\alpha,N,\lambda,\Lambda,M,\overline{M},s_1$.
  \end{itemize}
\end{theorem}

This paper is organized as follows. In Section \ref{sec1}, we first collect some basic notations, notions and give the hypotheses on \eqref{main}. Sections \ref{sec2} and \ref{sec3} are devoted to showing the $C^{1,\alpha}$-regularity properties for \eqref{main}  under the scenarios that $\sigma\in(\sigma_0,2)$ and $\sigma\in(1,2)$ separately. Finally, H\"{o}lder continuity of solutions to \eqref{main} is justified for $\sigma\in(0, 1]$ in Section \ref{sec4}.

\section{Preliminaries}
\label{sec1}

In this section, we give the definitions of viscosity solutions and present some structure conditions on Eq. \eqref{main}.

\subsection{Assumptions}
Throughout this paper, we always give the following main hypotheses:

\smallskip

\begin{itemize}
\item [($A_1$)] we suppose $u\in L^1_\sigma(\mathbb{R}^N)$;

\smallskip

\item [($A_2$)] The source term $f$ belongs to $C(B_1)\cap L^\infty(B_1)$;

\smallskip

\item [($A_3$)] The function $\Phi:B_1\times[0,\infty)\rightarrow[0,\infty)$ is continuous and there is a constant $\overline{M}\ge 1$ such that $\overline{M}^{-1}\le \Phi(x,1)\le \overline{M}$ for any $x\in B_1$;

\smallskip

\item [($A_4$)] For $\Phi$, there exist constants $0\le s_1\le s_2$ such that for every $x\in B_1$ the map $t\rightarrow \frac{\Phi(x,t)}{t^{s_1}}$ is almost increasing and the map $t\rightarrow \frac{\Phi(x,t)}{t^{s_2}}$ is almost decreasing with constant $M\ge1$;

\smallskip

\item [($A_5$)] Let $\{K_{ij}\}_{ij}\subset \mathcal{K}$ be a set of kernels such that there exist a collection of numbers $\{k_{ij}\}\subset[\lambda,\Lambda]$ and a modulus of continuity $\omega$ fulfilling
    $$
    \left|K_{ij}|x|^{N+\sigma}-k_{ij}\right|\le\omega(|x|) \quad\text{for } |x|\le1.
    $$
\end{itemize}

Here a function $u:[0,\infty)\rightarrow[0,\infty)$ is almost decreasing or almost increasing if there is a constant $M\ge1$ such that $u(\tau)\le Mu(t)$ or $u(t)\le Mu(\tau)$ respectively for $0\le t< \tau$. The last assumption above assures the nonlocal operator $\mathcal{I}_\sigma$ approaches a local fully nonlinear operator $F$ with uniform ellipticity as $\sigma$ goes to 2. Concerning the derivation of conditions ($A_3$), ($A_4$), we mainly adopt the same hypotheses on $\Phi$ introduced in \cite{IFB}. Additionally, from a variational perspective, H\"{a}st\"{o} and Ok \cite{HO22} considered the gradient H\"{o}lder continuity for local minimizers of the energy functional
$$
u\mapsto\int_{B_1}\varphi(x,|Du|)\,dx
$$
with $\varphi:B_1\times[0,\infty)\rightarrow[0,\infty)$, where the key assumptions on the integral density $\varphi$ are that there exist two numbers $p,q>1$ such that $t\rightarrow \frac{\Phi(x,t)}{t^{p}}$ is almost non-decreasing and $t\rightarrow \frac{\Phi(x,t)}{t^{q}}$ is almost non-increasing. In this respect, ($A_3$) and ($A_4$) in \eqref{main} are also reasonable. Finally, let us mention that the conditions ($A_3$), ($A_4$) cover some remarkable cases besides $\Phi(x,t)=t^p+a(x)t^q$ or $\Phi(x,t)=t^{p(x)}+a(x)t^{q(x)}$, such as
\begin{itemize}
	\item $\Phi(x,t)=t^p+a(x)t^p\log(e+t)$.
	
	\smallskip
	
	\item $\Phi(x,t)=\phi(t)+a(x)\varphi(t)$ with suitable $N$-functions $\phi,\varphi$, where the function $0\le a(x)\in C(B_1)$.
	
\end{itemize}

\subsection{Notions}

Next, we give the concepts of viscosity solutions and so-called approximated viscosity solutions. If not important or not confused, we omit the subscript $\sigma$ of the operator $\mathcal{I}$ in the rest of the work. For generality, we define the solutions of a variety of \eqref{main} as below:
\begin{equation}
\label{2-1}
-\Phi(x,|Du+\eta|)\mathcal{I}(u,x)=f(x)  \quad\text{in } B_1,
\end{equation}
where the vector $\eta\in\mathbb{R}^N$. In what follows, for $u:\mathbb{R}^N\rightarrow\mathbb{R}$ and $K\in\mathcal{K}$, we define
$$
I_{K}[\Omega](u,x)=C_{N,\sigma}\mathrm{P.V.}\int_\Omega(u(x+y)-u(x))K(y)\,dy
$$
with $\Omega\subseteq\mathbb{R}^N$ a measurable set. 

\begin{definition}[viscosity solutions]
\label{def1}
We call the function $u\in C(\overline{B}_1)\cap L^1_\sigma(\mathbb{R}^N)$ a viscosity supersolution (subsolution) to \eqref{2-1}, whenever for any $x_0\in B_1$ and any $\varphi(x)\in C^2(\mathbb{R}^N)$ such that $u-\varphi$ reaches a local minimum (maximum) at $x_0$, there holds that
$$
-\Phi(x_0,|D\varphi(x_0)+\eta|)\mathcal{I}^\delta(u,\varphi,x_0)\ge (\le) f(x_0),
$$
where we denote
$$
\mathcal{I}^\delta(u,\varphi,x_0)=\inf_j\sup_i\left(I_{K_{ij}}[B_\delta](\varphi,x_0)+I_{K_{ij}}[B^c_\delta](u,x_0)\right)
$$
with $B_\delta$ being a neighborhood of $x_0$ and $B^c_\delta$ meaning the complement of $B_\delta$ in $\mathbb{R}^N$. Here $\{K_{ij}\}_{ij}\subset \mathcal{K}$ is a set of kernels. A function $u$ is called a viscosity solution to \eqref{2-1} when it is viscosity supersolution and subsolution simultaneously.
\end{definition}

The following approximated viscosity solutions play a fundamental role in proving Theorem \ref{thm3}.

\begin{definition}[Approximated viscosity solutions]
\label{def2}
The function $u\in C(\overline{B}_1)\cap L^1_\sigma(\mathbb{R}^N)$ is called an approximated viscosity solution to \eqref{2-1}, 
if there exist sequences of functions $\{u_k\}\subset C(B_1)\cap L^1_\sigma(\mathbb{R}^N)$, of vectors $\{\eta_k\}\subset\mathbb{R}^N$, of numbers $\{d_k\}\subset\mathbb{R}^+$ and $\mu>0$, fulfilling that $\eta_k\rightarrow\eta$, $d_k\rightarrow0$, $u_k\rightarrow u$ locally uniformly in $B_1$ as $k\rightarrow\infty$ and $|u_k|\le \mu(1+|x|^{1+\alpha})$ with $1+\alpha\in(0,\sigma)$, such that $u_k$ is a viscosity solution of
\begin{equation}
\label{2-2}
-\Phi(x,|Du_k+\eta_k|+d_k)\mathcal{I}(u_k,x)=f(x)  \quad\text{in } B_1.
\end{equation}
In addition, we take $\eta_k=0$ for any $k\in\mathbb{N}$ if $\eta=0$, and $d_k$ satisfies $kd_k^{s_2}\rightarrow\infty$ when $0\le s_1\le s_2$.
\end{definition}

\section{$C^{1,\alpha}$-regularity under a smallness condition on $2-\sigma$}
\label{sec2}

With the precondition that $2-\sigma$ is small enough, we in this part are going to show the gradient H\"{o}lder continuity for viscosity solutions of \eqref{main} and establish a Schauder-type regularity estimate at the origin. To make use of the compactness method, now we first utilize the Ishii-Lions method and elliptic estimates ``in the direction of gradient" presented in \cite{BCCI12, BCI11} to prove the local Lipschitz continuity of solutions to problem \eqref{2-1}.

\subsection{Gradient H\"{o}lder regularity}

\begin{proposition}
\label{pro3-1}
Let $\sigma\in(1,2)$ and the assumptions $(A_1)$, $(A_3)$, $(A_4)$ be in force. Suppose that $u\in C(\overline{B}_1)$ is a viscosity solution to \eqref{2-1}. Then $u$ is locally Lipschitz continuous in $B_1$, that is, there exists a constant $C_{\rm lip}\ge1$ depending on $\|u\|_{L^\infty(B_1)}$, $\|u\|_{L^1_\sigma(\mathbb{R}^N)}$, $\|f\|_{L^\infty(B_1)},N,\lambda,\Lambda,s_1,\sigma,M,\overline{M}$, but not on $\eta$, such that
$$
|u(x)-u(y)|\leq C_{\rm lip}|x-y|  
$$
for every $x,y\in B_{\frac{1}{2}}$. Moreover, the Lipschitz constant $C_{\rm lip}$ is uniformly bounded as $\sigma\rightarrow2$.
\end{proposition}

\begin{proof}
We divide this proof into two diverse cases in which $|\eta|$ is large or small.

\textbf{Case 1.} $|\eta|$ is large. Assume $|\eta|>a_0$ for a number $a_0>0$ to be fixed later. First, we introduce some nonnegative and smooth functions. Let $\psi:\mathbb{R}^N\rightarrow \mathbb{R}$ fulfill $\psi\equiv0$ in $B_{\frac{1}{2}}$ and $\psi\equiv1$ in $B^c_{\frac{3}{4}}$, and moreover define
$$
h(x)=\left(\mathop{\rm{osc}}\limits_{B_1}u+1\right)\psi(x).
$$
Denote
\begin{equation*}
\omega(t)=\begin{cases}
t-\frac{1}{4}t^{1+\alpha}, \quad &t\in[0,t_0],\\[2mm]
\omega(t_0), \quad &t\in(t_0,\infty),
\end{cases}
\end{equation*}
where $\alpha\in(0,1)$ is a small enough number that will be selected, and $t_0\le\left(\frac{4}{1+\alpha}\right)^\frac{1}{\alpha}$ is a fixed number.

We proceed by doubling the variables. Construct the auxiliary functions
$$
\varphi(x,y)=L\omega(|x-y|)+h(y)
$$
and
$$
\Psi(x,y)=u(x)-u(y)-\varphi(x,y).
$$
Here the constant $L\ge1$ will be determined later. In fact, it is the Lipschitz constant of viscosity solution. We can see readily, from the continuity of $\Psi$, that $\Psi$ realizes its maximum at $(\overline{x},\overline{y})\in \overline{B}_1\times\overline{B}_1$. We shall show $\Psi(\overline{x},\overline{y})\le0$ by contradiction argument so that we get the Lipschitz continuity. If not, there holds that
\begin{align*}
0&<u(\overline{x})-u(\overline{y})-L\omega(|\overline{x}-\overline{y}|)-h(\overline{y})\\
&\le \mathop{\rm{osc}}\limits_{B_1}u-\left(\mathop{\rm{osc}}\limits_{B_1}u+1\right)\psi(\overline{y})-L\omega(|\overline{x}-\overline{y}|).
\end{align*}
Obviously, $\overline{x}\neq\overline{y}$. Due to the definition of $\psi$, we have $\overline{y}\in B_\frac{3}{4}$. Otherwise, there is a contradiction that
$$
0<\mathop{\rm{osc}}\limits_{B_1}u-\mathop{\rm{osc}}\limits_{B_1}u-1-L\omega(|\overline{x}-\overline{y}|)<0.
$$
Furthermore, we also have
\begin{equation}
\label{3-1}
\mathop{\rm{osc}}\limits_{B_1}u\ge L\omega(|\overline{x}-\overline{y}|)=L|\overline{x}-\overline{y}|\left(1-\frac{1}{4}|\overline{x}-\overline{y}|^\alpha\right)\ge\frac{L}{2}|\overline{x}-\overline{y}|,
\end{equation}
which indicates that
$$
|\overline{x}|\le|\overline{x}-\overline{y}|+|\overline{y}|\le\frac{2}{L}\mathop{\rm{osc}}\limits_{B_1}u+\frac{3}{4}\le \frac{7}{8}
$$
by choosing $L\ge16\mathop{\rm{osc}}\limits_{B_1}u$.

Next, we are going to establish the viscosity inequality of $u$, which is the base for getting contradiction. Since the function $\Psi$ attains its maximum in $\overline{B}_1\times\overline{B}_1$ at $(\overline{x},\overline{y})$, then $u(x)-\varphi(x,\overline{y})$ has the local maximum at $\overline{x}$ and $u(y)-(-\varphi(\overline{x},y))$ gets the local minimum at $\overline{y}$. Through the definition of viscosity solution, we obtain
\begin{equation}
\label{3-2}
\begin{cases}
-\Phi(\overline{x},|D_x\varphi(\overline{x},\overline{y})+\eta|)\mathcal{I}^\delta(u,\varphi(\cdot,\overline{y}),\overline{x})\le f(\overline{x})\\[2mm]
-\Phi(\overline{y},|-D_y\varphi(\overline{x},\overline{y})+\eta|)\mathcal{I}^\delta(u,-\varphi(\overline{x},\cdot),\overline{y})\ge f(\overline{y})
\end{cases}
\end{equation}
for every $\delta\in(0,1)$, where
$$
D_x\varphi(\overline{x},\overline{y})=L\omega'(|\overline{x}-\overline{y}|)\frac{\overline{x}-\overline{y}}{|\overline{x}-\overline{y}|} \quad \text{and}\quad
-D_y\varphi(\overline{x},\overline{y})=L\omega'(|\overline{x}-\overline{y}|)\frac{\overline{x}-\overline{y}}{|\overline{x}-\overline{y}|}-Dh(\overline{y}).
$$
Observe that
\begin{equation}
\label{3-3}
\frac{3}{16}L\le|D_x\varphi(\overline{x},\overline{y})|=L\left|1-\frac{1+\alpha}{4}|\overline{x}-\overline{y}|^\alpha\right|\le L
\end{equation}
and
\begin{equation}
\label{3-4}
\frac{L}{8}\le\frac{3}{16}L-|Dh(\overline{y})|\le|D_y\varphi(\overline{x},\overline{y})|=L+|Dh(\overline{y})|\le\frac{17}{16}L,
\end{equation}
where we have picked $L$ so large that
$$
\left(\mathop{\rm{osc}}\limits_{B_1}u+1\right)|Dh(\overline{y})|\le C(N)\left(\mathop{\rm{osc}}\limits_{B_1}u+1\right)\le\frac{L}{16}.
$$

Recall that we have assumed $|\eta|>a_0$. Now set $a_0=2L$. Notice that, from \eqref{3-3} and \eqref{3-4},
\begin{equation}
\label{3-5}
\begin{cases}
|D_x\varphi(\overline{x},\overline{y})+\eta|\ge a_0-L\ge L   \\[2mm]
|-D_y\varphi(\overline{x},\overline{y})+\eta|\ge a_0-\frac{17}{16}L\ge \frac{15}{16}L.
\end{cases}
\end{equation}
Therefore, it follows from the inequalities \eqref{3-2} and \eqref{3-5} that
\begin{equation*}
\begin{cases}
-\mathcal{I}^\delta(u,\varphi(\cdot,\overline{y}),\overline{x})\le\frac{\|f\|_{L^\infty(B_1)}}{\Phi(\overline{x},|D_x\varphi(\overline{x},\overline{y})+\eta|)}
\le \frac{C\|f\|_{L^\infty(B_1)}}{|D_x\varphi(\overline{x},\overline{y})+\eta|^{s_1}}\le \frac{C\|f\|_{L^\infty(B_1)}}{L^{s_1}},  \\[2mm]
-\mathcal{I}^\delta(u,-\varphi(\overline{x},\cdot),\overline{y})\ge-\frac{C\|f\|_{L^\infty(B_1)}}{|-D_y\varphi(\overline{x},\overline{y})+\eta|^{s_1}}\ge -\frac{C\|f\|_{L^\infty(B_1)}}{L^{s_1}}.
\end{cases}
\end{equation*}
Here we utilized the connection that $\overline{M}^{-1}\le \Phi(x,1)\le M\frac{\Phi(x,t)}{t^{s_1}}$ for $t\ge 1$. Then it yields that
$$
\mathcal{I}^\delta(u,\varphi(\cdot,\overline{y}),\overline{x})-\mathcal{I}^\delta(u,-\varphi(\overline{x},\cdot),\overline{y})\ge-C\|f\|_{L^\infty(B_1)}
$$
with some universal constant $C\ge1$. At this moment, we want to estimate the left-hand side of the last display that consists of uniformly elliptic nonlocal operators. However, this procedure is identical to that in the proof of \cite[Lemma 3.1]{APS24} or \cite[Lemma 2.2]{PT21}, so we do not repeat the process anymore and directly give the final estimate
\begin{equation}
\label{3-6}
-C\left(\|f\|_{L^\infty(B_1)}+\mathop{\rm{osc}}\limits_{B_1}u+\|u\|_{L^1_\sigma(\mathbb{R}^N)}+1\right)\le -CL|\overline{x}-\overline{y}|^{1-\sigma+\alpha(N+2-\sigma)}.
\end{equation}
Remembering $|\overline{x}-\overline{y}|\le2L^{-1}\mathop{\rm{osc}}\limits_{B_1}u$ in \eqref{3-1} and $L>2\mathop{\rm{osc}}\limits_{B_1}u$, together with taking $\alpha>0$ sufficiently small to derive
$$
\overline{\alpha}:=1-\sigma+\alpha(N+2-\sigma)\le \frac{1-\sigma}{2}<0,
$$
it then follows from \eqref{3-6} that
$$
-C\left(\|f\|_{L^\infty(B_1)}+\mathop{\rm{osc}}\limits_{B_1}u+\|u\|_{L^1_\sigma(\mathbb{R}^N)}+1\right)\le-L,
$$
where $C>0$ is a universal constant. Apparently, we get a contradiction via choosing $L=2C\left(\|f\|_{L^\infty(B_1)}+\|u\|_{L^\infty(B_1)}+\|u\|_{L^1_\sigma(\mathbb{R}^N)}+1\right)$, which implies the assumption $\Psi(\overline{x},\overline{y})>0$ is false. We now complete the proof in the case of $|\eta|$ large.

\textbf{Case 2.} $|\eta|$ is small. Let us study the case $|\eta|\le a_0$ with $a_0$ being a fixed constant from Case 1. The procedure is very similar to Case 1, so we sketch it. Introduce a function
$$
\varphi(x,y)=\tilde{L}\omega(|x-y|)+h(y)
$$
and follow these notations in Case 1 except the number $L$ replacing by $\tilde{L}$. Then the inequalities \eqref{3-3} and \eqref{3-4} become
$$
\frac{\tilde{L}}{8}\le |D_x\varphi(\overline{x},\overline{y})|,|D_y\varphi(\overline{x},\overline{y})|\le \frac{17}{16}\tilde{L}.
$$
Since $|\eta|\le a_0$, we have
\begin{equation*}
\begin{cases}
|D_x\varphi(\overline{x},\overline{y})+\eta|\ge|D_x\varphi(\overline{x},\overline{y})|-|\eta|\ge\frac{\tilde{L}}{8}-a_0\ge a_0\\[2mm]
|-D_y\varphi(\overline{x},\overline{y})+\eta|\ge\frac{17\tilde{L}}{16}-a_0\ge 16a_0
\end{cases}
\end{equation*}
by picking $\tilde{L}\ge16a_0$. Moreover, we also arrive at
$$
\mathcal{I}^\delta(u,\varphi(\cdot,\overline{y}),\overline{x})-\mathcal{I}^\delta(u,-\varphi(\overline{x},\cdot),\overline{y})\ge-\frac{C\|f\|_{L^\infty(B_1)}}{a^{s_1}_0}.
$$
At this stage, following the same proof in Case 1 deduces the Lipschitz continuity of $u$ by determining $\tilde{L}$ large enough.
\end{proof}

\begin{lemma}
\label{lem3-2}
Let the conditions $(A_1)$--$(A_5)$ hold true with $\overline{M}=1$. Suppose that $u\in C(\overline{B}_1)$ is a viscosity solution of \eqref{2-1} with $\|u\|_{L^\infty(B_1)}\le1$. Given $\mu,\varepsilon>0$, there is a positive $\kappa$ depending on $N,\lambda,\Lambda,\varepsilon,s_1,s_2,M,\mu$ 
such that if
$$
|\sigma-2|+\|f\|_{L^\infty(B_1)}\leq \kappa
$$
and
$$
|u(x)|\le\mu(1+|x|^{1+\overline{\alpha}}) \quad  \text{for }  x\in\mathbb{R}^N
$$
with some $\overline{\alpha}\in(0,1)$, then there exists an $F$-harmonic function $h\in C^{1,\overline{\alpha}}_{\rm loc}(B_1)$ (i.e., $h$ solves $F(D^2h)=0$ in the viscosity sence) fulfilling
$$
\|u-h\|_{L^\infty(B_{1/2})}\leq \varepsilon
$$
\end{lemma}

\begin{proof}
Argue by contradiction. If not, then there are $\mu_0,\varepsilon_0>0$ and sequences of $\{\sigma_k\},\{f_k\}$, $\{\Phi_k\},\{u_k\}$ and a sequence of vectors $\{\eta_k\}$ such that
\begin{itemize}
  \item[(i)]
$$
|\sigma_k-2|+\|f_k\|_{L^\infty(B_1)}\leq \frac{1}{k};
$$

    \smallskip

  \item[(ii)] $u_k\in C(B_1)$ with $\|u_k\|_{L^\infty(B_1)}\leq 1$ and $|u_k(x)|\le\mu_0(1+|x|^{1+\overline{\alpha}})$ solves the following equation
      \begin{equation}
      \label{3-2-1}
      -\Phi_k(x,|Du_k+\eta_k|)\mathcal{I}_{\sigma_k}(u_k,x)=f_k(x)  \quad \text{in } B_1,
      \end{equation}
      where the operator $\mathcal{I}_{\sigma_k}$ satisfies $(A_5)$.

    \smallskip

  \item[(iii)] Here $\Phi_k\in C(B_1\times[0,\infty),[0,\infty))$ satisfies that the function $t\rightarrow\frac{\Phi_k(x,t)}{t^{s_1}}$ is almost increasing and the function $t\rightarrow\frac{\Phi_k(x,t)}{t^{s_2}}$ is almost decreasing with the same constant $M\ge1$ and $\Phi_k(x,1)=1$ for all $x\in B_1$.

\end{itemize}
Nonetheless, we have
$$
\|u_k-h\|_{L^\infty(B_{1/2})}>\varepsilon_0
$$
for any $h(x)\in C^{1,\overline{\alpha}}_{\rm loc}(B_1)$. Recalling $\sigma_k\rightarrow2$, by means of the condition ($A_5$), we know that $\mathcal{I}_{\sigma_k}\rightarrow F$, where $F$ is a uniformly $(\lambda,\Lambda)$-elliptic operator. Moreover, it follows from Proposition \ref{pro3-1} that there is a continuous function $\overline{u}$ such that $u_k\rightarrow \overline{u}$ locally uniformly in $B_1$. Particularly, it holds that
$$
\overline{u}\in C(B_{3/4}) \quad \text{and} \quad \|\overline{u}\|_{L^\infty(B_{3/4})}\leq1,
$$
but
\begin{equation}
\label{3-2-2}
\sup_{x\in B_{1/2}}|\overline{u}(x)-h(x)|>\varepsilon_0.
\end{equation}

In what follows, we are going to demonstrate that $\overline{u}$ solves the equation below in the viscosity sense
\begin{equation}
\label{3-2-3}
-F(D^2\overline{u})=0   \quad \text{in }  B_{3/4}.
\end{equation}
To this end, we only prove $\overline{u}$ is a viscosity supersolution because the result that $\overline{u}$ is a subsolution can be shown similarly. Suppose that $\varphi(x)$ is a test function such that $\overline{u}-\varphi$ realizes its local minimum in $B_1$ at $\tilde{x}$. Without loss of generality, assume $|\tilde{x}|=\overline{u}(0)=\varphi(0)=0$ and $\varphi$ is a quadratic polynomial, that is,
$$
\varphi(x)=\frac{1}{2}Ax\cdot x+p\cdot x,
$$
where $p\in\mathbb{R}^N$ and $A$ is a symmetric matrix and $Ax\cdot x$ or $p\cdot x$ denotes the inner product. Thanks to $u_k\rightarrow\overline{u}$ locally uniformly in $B_1$, we could find a point $x_k$ and a quadratic polynomial
$$
\varphi_k(x):=\frac{1}{2}A(x-x_k)\cdot(x-x_k)+p\cdot(x-x_k)+u_k(x_k)
$$
touching $u_k$ from below at $x_k$ lying in a small neighborhood of the origin. In view of $u_k$ a viscosity solution to \eqref{3-2-1}, it yields that
\begin{equation}
\label{3-2-4}
 -\Phi_k(x_k,|p+\eta_k|)\mathcal{I}^\delta_{\sigma_k}(u_k,\varphi_k,x_k)\ge f_k(x_k).
\end{equation}

Next, we distinguish two different scenarios. First of all, if $\{\eta_k\}$ is unbounded, then, with the aid of $\Phi_k(x_k,|p+\eta_k|)\ge C|p+\eta_k|^{s_1}\rightarrow\infty$ up to a subsequence, we pass to the limit in \eqref{3-2-4} and get
$$
-F(A)=-\lim_{k\rightarrow\infty}\mathcal{I}^\delta_{\sigma_k}(u_k,\varphi_k,x_k)\ge \lim_{k\rightarrow\infty}\frac{f_k(x_k)}{\Phi_k(x_k,|p+\eta_k|)}=0.
$$
For the second case that $\{\eta_k\}$ is bounded, we may suppose $\eta_k\rightarrow\overline{\eta}$ (up to a subsequence). We first examine that if $|p+\overline{\eta}|\neq0$, then $|p+\eta_k|\ge\frac{1}{2}|p+\overline{\eta}|$ for $k$ large enough. Thereby, applying the assumptions on $\Phi_k$, (for large $k$) we have
\begin{equation*}
\begin{cases}
\Phi_k(x_k,|p+\eta_k|)\ge M^{-1}|p+\eta_k|^{s_1}\ge2^{-s_1}M^{-1}|p+\overline{\eta}|^{s_1}, \ \text{for } |p+\overline{\eta}|\ge1 \\[2mm]
\Phi_k(x_k,|p+\eta_k|)\ge M^{-1}|p+\eta_k|^{s_2}\ge2^{-s_2}M^{-1}|p+\overline{\eta}|^{s_2}, \ \text{for } |p+\overline{\eta}|<1
\end{cases}
\end{equation*}
and further get
$$
-F(A)=-\lim_{k\rightarrow\infty}\mathcal{I}^\delta_{\sigma_k}(u_k,\varphi_k,x_k)\ge -\lim_{k\rightarrow\infty}\frac{2^{s_2}M}{k\min\{|p+\overline{\eta}|^{s_1},|p+\overline{\eta}|^{s_2}\}}=0.
$$

Now let us concentrate on justifying $F(A)\le0$ under the situation $|p+\overline{\eta}|=0$ that involves $p=\overline{\eta}=0$ or $p=-\overline{\eta}\neq0$. Suppose by contradiction
\begin{equation}
\label{3-2-4-1}
F(A)>0.
\end{equation}
Thus we can find from the uniform ellipticity of $F(\cdot)$ that the matrix $A$ has one positive eigenvalue at least. Let $\mathbb{R}^n=T\oplus Q$ be the orthogonal sum, where $T$ 
stands for the invariant space made out of the eigenvectors related to positive eigenvalues.

\medskip

\textbf{Case 1.} $b=-\overline{\eta}\neq0$. Let $\gamma>0$ and
$$
\phi_\gamma(x):=\varphi(x)+\gamma|P_T(x)|=\frac{1}{2}Ax\cdot x+p\cdot x+\gamma|P_T(x)|
$$
with $P_T$ meaning the orthogonal projection over $T$. Owing to $u_k\rightarrow \overline{u}$ locally uniformly and $\varphi(x)$ touching $\overline{u}(x)$ from below at 0, then for $\gamma$ sufficiently small $\phi_\gamma(x)$ touches $u_k(x)$ from below at some point $x^\gamma_k$ in a neighbourhood of 0. Besides, $x^\gamma_k\rightarrow\overline{x}$ for some $\overline{x}\in B_{3/4}$ as $k\rightarrow\infty$ (up to a subsequence).

When $P_T(x^\gamma_k)=0$, we announce $F(A)\leq 0$, which contradicts \eqref{3-2-4-1}. Notice that
$$
\phi_\gamma(x):=\frac{1}{2}Ax\cdot x+p\cdot x+\gamma e\cdot P_T(x)
$$
touches $u_k$ from below at $x^\gamma_k$ for all $e\in \mathbb{S}^{N-1}$ (i.e., $|e|=1$). Then we have the viscosity inequality as
\begin{equation}
\label{3-2-5}
-\Phi_k(x^\gamma_k,|\eta_k+Ax^\gamma_k+p+\gamma P_T(e)|)\mathcal{I}^\delta_{\sigma_k}(u_k,\phi_\gamma,x^\gamma_k)\ge f_k(x^\gamma_k),
\end{equation}
where we note $D(e\cdot P_T(x))=P_T(e)$. Now choosing $e\in T\cap \mathbb{S}^{N-1}$ yields $P_T(e)=e$. If $A\overline{x}=0$, then it follows from $|\eta_k+p|\rightarrow0$ that, for large $k$,
\begin{equation*}
|\eta_k+p+Ax^\gamma_k+\gamma e|>\frac{\gamma}{2}.
\end{equation*}
Therefore, via the structure on $\Phi_k$, we get
$$
\lim_{k\rightarrow\infty}\frac{|f_k(x^\gamma_k)|}{\Phi_k(x^\gamma_k,|\eta_k+Ax^\gamma_k+p+\gamma P_T(e)|)}\le\lim_{k\rightarrow\infty}\frac{C}{k\min\{\gamma^{s_1},\gamma^{s_2}\}}=0,
$$
and further have $F(A)\le0$ by \eqref{3-2-5}. If $A\overline{x}\neq0$, then in the case $T\equiv \mathbb{R}^N$ we pick carefully $e\in \mathbb{S}^{N-1}$ such that
$$
|\eta_k+p+Ax^\gamma_k+\gamma e|\ge\frac{1}{2}|A\overline{x}+\gamma e|-\frac{1}{4}|A\overline{x}+\gamma e|=\frac{1}{4}|A\overline{x}+\gamma e|>0.
$$
Hence, we can obtain $F(A)\le0$. On the other hand, if $A\overline{x}\neq0$ and $T\neq\mathbb{R}^N$, we take $e\in T^\bot\cap \mathbb{S}^{N-1}$ so that
$$
|\eta_k+p+Ax^\gamma_k+\gamma P_T(e)|\ge\frac{1}{2}|A\overline{x}|-\frac{1}{4}|A\overline{x}|=\frac{1}{4}|A\overline{x}|>0,
$$
where we observe $P_T(e)=0$. Thus, analogous to the case $A\overline{x}=0$ again, we can see $F(A)\le0$.

Finally, let us treat the occurrence $P_T(x^\gamma_k)\neq0$. We can find that the map $x\rightarrow|P_T(x)|$ is convex and smooth near the point $x^\gamma_k$. Because of $P_T$ being a projection, then
\begin{equation}
\label{3-2-6}
|P_T(x)|D(|P_T(x)|)=P_T(x) \ \text{ and } \ D^2(|P_T(x)|) \ \text{is nonnegative definite}.
\end{equation}
Hence we arrive at the following viscosity inequality
\begin{align*}
-\Phi_k(x^\gamma_k,|\eta_k+Ax^\gamma_k+p+\gamma \hat{e}|)\mathcal{I}^\delta_{\sigma_k}(u_k,\phi_\gamma,x^\gamma_k)\ge f_k(x^\gamma_k),
\end{align*}
where we let $\hat{e}=\frac{P_T(x^\gamma_k)}{|P_T(x^\gamma_k)|}$ for simplicity. Considering separately the scenarios that $A\overline{x}=0$ and $A\overline{x}\neq0$ like the case $P_T(x^\gamma_k)=0$ leads to
$$
-F(A+D^2|P_T(\overline{x})|)\ge0.
$$
By virtue of \eqref{3-2-6} and the ellipticity condition on $F$, we derive $F(A)\le0$, which contradicts \eqref{3-2-4-1}.

\medskip

\textbf{Case 2.} $p=\eta=0$. At this point, the procedures become easier. Since $\frac{1}{2}Ax\cdot x$ touches $\overline{u}(x)$ from below at 0 and $u_k\rightarrow\overline{u}$ locally uniformly, then the test function
$$
\hat{\phi}_\gamma(x):=\frac{1}{2}Ax\cdot x+\gamma|P_T(x)|
$$
touches $u_k$ from below at some $\hat{x}_k$ in a small neighborhood of 0 for $k$ large enough. Likewise, we will check two cases that $|P_T(\hat{x}_k)|=0$ and $|P_T(\hat{x}_k)|>0$, and further estimate the boundedness on $|\eta_k+A\hat{x}_k+e|$ as well as $|\eta_k+A\hat{x}_k+\gamma \hat{e}|$ with $\hat{e}:=\frac{P_T(\hat{x}_k)}{|P_T(\hat{x}_k)|}$ $(P_T(\hat{x}_k)\neq0)$, which is in a similar manner to Case 1. Eventually, we conclude $F(A)\leq0$, which contradicts \eqref{3-2-4}.

At this stage, we have proved $\overline{u}$ is a viscosity supersolution to \eqref{3-2-3} and in a specular way we could verify it is a subsolution as well. It is well known in \cite[Chapter 5]{CC95} that the solution, $\overline{u}$, of \eqref{3-2-3} is of $C^{1,\overline{\alpha}}$ locally for some $\overline{\alpha}\in (0,1)$, so we choose $h=\overline{u}$ and obtain a contradiction with \eqref{3-2-2}. Now the proof is finished.
\end{proof}

\begin{lemma}
\label{lem3-3}
Let the preconditions $(A_1)$--$(A_5)$ be fulfilled. Suppose $u$ is a normalized viscosity solution of \eqref{main}. Given $\mu>0$, there is a $\kappa>0$ depending on $N,\lambda,\Lambda,s_1,s_2,M,\overline{M},\alpha$ such that if
$$
|\sigma-2|+\|f\|_{L^\infty(B_1)}\leq \kappa
$$
and
$$
|u(x)|\le\mu(1+|x|^{1+\overline{\alpha}}) \quad \text{for } x\in\mathbb{R}^N
$$
with $\overline{\alpha}$ coming from Lemma \ref{lem3-2}, then we can find a $\rho\in\left(0,\frac{1}{2}\right)$, depending on $N,\lambda,\Lambda,\alpha$, and a sequence of affine functions $\{l_n\}$ of the type $l_n(x)=a_n+p_n\cdot x$ for which
\begin{equation}
\label{3-3-0}
\|u-l_n\|_{L^\infty(B_{\rho^n})}\leq \rho^{n(1+\alpha)}
\end{equation}
with
$$
|a_{n+1}-a_n|+\rho^n|p_{n+1}-p_n|\leq C\rho^{n(1+\alpha)}
$$
for every
$$
\alpha\in(0,\overline{\alpha})\cap \left(0,\frac{\sigma-1}{1+s_2}\right].
$$
Here $C\ge1$ only depends on $N,\lambda,\Lambda$.
\end{lemma}

\begin{proof}
\textbf{Step 1}. First, we assert that there is an affine function $l(x)$ and a $\rho\in\Big(0,\frac{1}{2}\Big)$ such that
 $$
 \sup_{x\in B_\rho}|u-l|\le\rho^{1+\alpha}.
 $$
Assume with no loss of generality $\Phi(x,1)=1$. Otherwise, set $\overline{\Phi}(x,t)=\frac{\Phi(x,t)}{\Phi(x,1)}$. Let $h\in C^{1,\overline{\alpha}}(B_{3/4})$ is an $F$-harmonic function from Lemma \ref{lem3-2} such that
$$
\|u-h\|_{L^\infty(B_{3/4})}\le\varepsilon
$$
with $\varepsilon>0$ to be fixed a posteriori. The existence of such function $h(x)$ is assured, provided $\kappa>0$ is small enough from Lemma \ref{lem3-2}.

It follows from the regularity theory for $h$ in \cite{CC95} that for $\rho\in(0,1)$
$$
\sup_{x\in B_\rho}|h(x)-(h(0)+Dh(0)\cdot x)|\leq C\rho^{1+\overline{\alpha}},
$$
and
$$
|h(0)|+|Dh(0)|\leq C,
$$
where both $C$ and $\overline{\alpha}\in(0,1)$ depend only on $N,\lambda,\Lambda$. Denote
$$
l(x)=a+p\cdot x:=h(0)+Dh(0)\cdot x.
$$
Then
\begin{align*}
\sup_{x\in B_\rho}|u(x)-l(x)|\leq \sup_{x\in B_\rho}|u(x)-h(x)|+\sup_{x\in B_\rho}|h(x)-l(x)|<\varepsilon+C\rho^{1+\overline{\alpha}}.
\end{align*}
Since $0<\alpha<\overline{\alpha}$, we take $0<\rho\ll1$ satisfying
$$
C\rho^{\overline{\alpha}-\alpha}\leq\frac{1}{2} \quad \Rightarrow \rho\leq(2C)^{-\frac{1}{\overline{\alpha}-\alpha}}.
$$
In addition, fix $\varepsilon=\frac{1}{2}\rho^{1+\alpha}$. 
This leads to the claim.

\medskip

\textbf{Step 2}. We proceed by induction. For $k=1$, this is the content in Step 1. Suppose the conclusion in Lemma \ref{lem3-3} holds for $k=1,2,\cdots,n$. Now we shall justify that for $k=n+1$. Introduce a function $u_n(x):\mathbb{R}^N\rightarrow\mathbb{R}$ as
$$
u_n(x):=\frac{u(\rho^n x)-l_n(\rho^n x)}{\rho^{n(1+\alpha)}}.
$$
We can check that $u_n$ solves in the viscosity sense an equation of the form
$$
-\overline{\Phi}(x,|Du_n+\rho^{-n\alpha}p_n|)\overline{\mathcal{I}}(u_n,x)=\overline{f}(x),
$$
where the nonlocal operator $\overline{\mathcal{I}}$ has the same uniform ellipticity condition as the operator $\mathcal{I}$ in \eqref{main} (more details on this can be seen in \cite{CS11}), and
$$
\overline{\Phi}(x,t)=\frac{\Phi(\rho^nx,\rho^{n\alpha}t)}{\Phi(\rho^nx,\rho^{n\alpha})} \quad\text{and}\quad \overline{f}(x)=\frac{\rho^{n\sigma-n(1+\alpha)}}{\Phi(\rho^nx,\rho^{n\alpha})}f(\rho^nx).
$$
Observe that $\overline{\Phi}(x,1)=1$ and the function $t\rightarrow\frac{\overline{\Phi}(x,t)}{t^{s_1}}$ is almost non-decreasing and the function $t\rightarrow\frac{\overline{\Phi}(x,t)}{t^{s_2}}$ is almost non-increasing with the identical constant $M\ge1$ in $(A_4)$. Recall $\alpha\le \frac{\sigma-1}{1+s_2}$. Then we have $0\le \sigma-1-\alpha(1+s_2)$ and further
$$
\frac{\rho^{n(\sigma-1-\alpha)}}{\Phi(\rho^nx,\rho^{n\alpha})}\le\frac{\rho^{n(\sigma-1-\alpha)}}{M^{-1}\rho^{ns_2\alpha}}\le M,
$$
which indicates $\|\overline{f}\|_{L^\infty(B_1)}\le M\varepsilon$. Once we justify the assertion
\begin{equation}
\label{3-3-1}
|u_n(x)|\le1+|x|^{1+\overline{\alpha}} \quad\text{for } x\in\mathbb{R}^N,
\end{equation}
we could apply the conclusion in Step 1 to arrive at
$$
\sup_{B_\rho}|u_n(x)-\tilde{l}(x)|\le \rho^{1+\alpha},
$$
where $\tilde{l}(x)$ is an affine function of the form $\tilde{l}(x)=\tilde{a}+\tilde{p}\cdot x$ with $|\tilde{a}|+|\tilde{p}|\le C(N,\lambda,\Lambda)$. Scaling back, it holds that
$$
\sup_{B_{\rho^{n+1}}}|u(x)-l_{n+1}(x)|\le \rho^{(n+1)(1+\alpha)}
$$
with
$$
l_{n+1}:=a_{n+1}+p_{n+1}\cdot x=l_n(x)+\rho^{n(1+\alpha)}\tilde{l}(\rho^{-n}x).
$$
Here we deduce
$$
|a_{n+1}-a_n|\le C\rho^{n(1+\alpha)} \quad \text{and}\quad |p_{n+1}-p_n|\le C\rho^{n\alpha},
$$
as desired.

We eventually verify \eqref{3-3-1} by induction to complete the proof. For $k=0$, take $u_0=u$. Assume the inequality \eqref{3-3-1} holds for $k=0,1,2,\cdots,n$. Next, we prove it for $k=n+1$. Observe that
\begin{align*}
u_{n+1}(x)&=\rho^{-(1+\alpha)}\left[\frac{u(\rho^n(\rho x))-l_n(\rho^n(\rho x))}{\rho^{n(1+\alpha)}}-\tilde{l}(\rho x)\right]\\
&=\frac{u_n(\rho x)-\tilde{l}(\rho x)}{\rho^{1+\alpha}}.
\end{align*}
When $\rho|x|>\frac{1}{2}$, we continue to take $\rho\le \Big(\frac{1}{10(1+C)}\Big)^\frac{1}{\overline{\alpha}-\alpha}$ to
\begin{align*}
|u_{n+1}(x)|&\le \rho^{-(1+\alpha)}[(1+|\rho x|^{1+\overline{\alpha}})+C(1+|\rho x|)]\\
&\le\rho^{\overline{\alpha}-\alpha}(5+6C)|x|^{1+\overline{\alpha}}\\
&\le|x|^{1+\overline{\alpha}}.
\end{align*}
On the other hand, when $\rho|x|\le\frac{1}{2}$, we get
\begin{align*}
|u_{n+1}(x)|&\le \rho^{-(1+\alpha)}(|u_n(\rho x)-h(\rho x)|+|h(\rho x)-\tilde{l}(\rho x)|)\\
&\le\rho^{-(1+\alpha)}(\rho^{1+\alpha}/2+C\rho^{1+\overline{\alpha}}|x|^{1+\overline{\alpha}})\\
&\le\frac{1}{2}+C\rho^{\overline{\alpha}-\alpha}|x|^{1+\overline{\alpha}}\\
&\le1+|x|^{1+\overline{\alpha}},
\end{align*}
where $h$ is from Lemma \ref{lem3-2}. Finally, we select $\rho=\frac{1}{2}\min\left\{\Big(\frac{1}{2C}\Big)^\frac{1}{\overline{\alpha}-\alpha},\Big(\frac{1}{10(1+C)}\Big)^\frac{1}{\overline{\alpha}-\alpha}\right\}$, and complete this proof.
\end{proof}

To show Theorem \ref{thm1}, we will exploit Lemma \ref{lem3-3} and make use of the scaling features of \eqref{main} to trace the problem back to a smallness regime. That is, it is possible to suppose that
\begin{equation}
\label{t1}
\|u\|_{L^\infty(B_1)}\le1 \quad\text{and} \quad \|f\|_{L^\infty(B_1)}\le\varepsilon
\end{equation}
with $0<\varepsilon\ll1$. The viscosity solution $u$ to \eqref{main} with \eqref{t1} is termed a normalized solution. Now checking its scaling properties permits us to work under the hypothesis \eqref{t1}. Define $v:\mathbb{R}^N\rightarrow\mathbb{R}$ as
$$
v(x):=\frac{u(x_0+rx)}{K},
$$
where $r\in(0,1)$ satisfies $B_r(x_0)\subset B_1$ and $K\ge1$ is a number to be fixed. It is easy to find that $v$ solves in the viscosity sense
\begin{equation}
\label{t2}
-\overline{\Phi}(x,|Dv|)\overline{\mathcal{I}}(v,x)=\overline{f}(x) \quad\text{in }  B_1,
\end{equation}
where $\overline{\mathcal{I}}$ is a uniformly elliptic nonlocal operator of the same ellipticity type as $\mathcal{I}$ in \eqref{main}. Here 
$$
\overline{f}(x)=\frac{r^\sigma f(x_0+rx)}{K\Phi\left(x_0+rx,\frac{K}{r}\right)}
$$
and
$$
\overline{\Phi}(x,|Dv|)=\frac{\Phi\left(x_0+rx,\frac{K}{r}|Dv|\right)}{\Phi\left(x_0+rx,\frac{K}{r}\right)}.
$$
We can know that $\overline{\Phi}(x,1)=1$ for $x\in B_1$ and $\overline{\Phi}(x,t)$ fulfills the structure conditions in $(A_4)$ with the same constants there. Furthermore, $(A_3)$--$(A_4)$ ensure
$$
\|\overline{f}\|_{L^\infty(B_1)}\le\frac{r^{s_1+\sigma}\|f\|_{L^\infty(B_1)}}{M^{-1}\overline{M}^{-1}K^{1+s_1}}\le \varepsilon,
$$
via choosing
$$
K=1+\|u\|_{L^\infty(B_1)}+\left(\frac{M\overline{M}\|f\|_{L^\infty(B_1)}}{\varepsilon}\right)^\frac{1}{1+s_1}.
$$
As a result, $v(x)$ solves the equation \eqref{t2} in the same class as \eqref{main} with the small regime in \eqref{t1}.

At this moment, the assumptions in Lemma \ref{lem3-3} have been fulfilled so that we can apply this lemma to conclude the main result. Notice that both of the sequences $\{a_n\},\{p_n\}$ converge, because they are Cauchy sequences. Thus we take an affine function $l_\infty(x)$ to establish the inequality
\begin{equation}
\label{t3}
\sup_{B_r}|u(x)-l_\infty(x)|\le Cr^{1+\alpha}
\end{equation}
for all $r\in(0,\rho]$, which is obtained by using the discrete iteration display \eqref{3-3-0} and implies the $C^{1,\alpha}$-regularity of $u$. Here the $C\ge1$ is a universal constant. The evaluation on \eqref{t3} is very standard, the details of which can be found for instance \cite{FRZ21, APS24}.


\subsection{Sharp regularity}

This subsection is devoted to establishing an improved gradient estimate at the origin. To facilitate the presentation, we focus on the special case that $\Phi(x,t)=t^p$. Now consider the following equation
\begin{equation}
\label{3-4-1}
|Du|^p\mathcal{I}(u,x)=f(x) \quad\text{in } B_1.
\end{equation}
Throughout this portion, we assume the source term $f$ on the right-hand side of \eqref{3-4-1} satisfies
\begin{equation}
\label{3-4-2}
|f(x)|\le \mathfrak{m}|x|^\theta  \quad\text{for } x\in B_1
\end{equation}
with $\theta\in(0,1)$, and we should keep in mind that the solutions of \eqref{3-4-1} is locally of class $C^{1,\alpha}(B_1)$ as soon as $\sigma$ is sufficiently close to 2.

\begin{lemma}
\label{lem3-4}
Let the conditions $(A_1),(A_2),(A_5)$ be in force. Let $u\in C(\overline{B}_1)$ be a normalized viscosity solution of Eq. \eqref{3-4-1} with $u(0)=0$. Given $\mu,\varepsilon>0$, there is a constant $\kappa$ depending on $N,\lambda,\Lambda,p,\varepsilon,\mu$ such that whenever
$$
|Du(0)|+|\sigma-2|+\|f\|_{L^\infty(B_1)}\le \kappa
$$
and
$$
|u(x)|\le \mu(1+|x|^{1+\overline{\alpha}})
$$
for $x\in\mathbb{R}^N$ with some $\overline{\alpha}\in(0,1)$, then we can find an $F$-harmonic function $h$ with $h\in C^{1,\overline{\alpha}}_{\rm loc}(B_1)$, 
so that
$$
0\in \mathcal{C}(h):=\{x\in B_1|h(x)=|Dh(x)|=0\} \quad \text{and}\quad \|u-h\|_{L^\infty(B_{3/4})}\le \varepsilon.
$$
\end{lemma}

\begin{proof}
Argue by contradiction. If this assertion is false, then we can assume that there exist $\mu_0,\varepsilon_0>0$ and sequences of functions $\{u_k\},\{f_k\}$, and a sequence of numbers $\{\sigma_k\}$ such that
$$
|Du_k|^p\mathcal{I}_{\sigma_k}(u_k,x)=f_k(x) \quad\text{in } B_1
$$
with $u_k(0)=0$ and $\|u_k\|_{L^\infty(B_1)}\le1$, as well as
$$
|Du_k(0)|+|\sigma_k-2|+\|f_k\|_{L^\infty(B_1)}\le \frac{1}{k},
$$
where the nonlocal operator $\mathcal{I}_{\sigma_k}$ fulfills the condition $(A_5)$. However, it holds
\begin{equation}
\label{3-4-3}
\|u_k-h\|_{L^\infty(B_{3/4})}>\varepsilon_0
\end{equation}
for all $F$-harmonic functions $h$ with $0\in\mathcal{C}(h)$. Indeed, since the equation $F(D^2h)=0$ does not has explicit dependency on $h$ and $Dh$, then we may suppose, up to a transform $v(x)=h(x)-h(0)-Dh(0)\cdot x$, $h(0)=|Dh(0)|=0$.

With the help of Theorem \ref{thm1} and Arzel\`{a}-Ascoli theorem, we know that $\{u_k\}$ converges (up to a subsequence) locally uniformly to a continuous function $\overline{u}$ in $B_1$ in the $C^{1}$-topology. That is, it yields that
\begin{equation}
\label{3-4-4}
\overline{u}(0)=|D\overline{u}(0)|=0.
\end{equation}
Utilizing the condition $(A_5)$ and the fact $\sigma_k\rightarrow 2$, we can see that $\mathcal{I}_{\sigma_k}\rightarrow F$ with $F$ being a uniformly $(\lambda,\Lambda)$-elliptic operator. Furthermore, there holds $f_k\rightarrow0$ uniformly in $B_1$. The remaining procedures are almost the same as that of Lemma \ref{lem3-2}. Then we can get the result that $\overline{u}$ solves in the viscosity sense $F(D^2\overline{u})=0$ in $B_{7/8}$. From this and \eqref{3-4-4}, we arrive at a contradiction with \eqref{3-4-3} for large enough $k$ by selecting $h=\overline{u}$.
\end{proof}

\begin{lemma}
\label{lem3-5}
Under the hypotheses of Lemma \ref{lem3-4} above, given a $\mu>0$, there is a $\kappa>0$ depending on $N,\lambda,\Lambda,p,\alpha,\mu$ and a $\rho\in\left(0,\frac{1}{2}\right)$ depending only on $N,\lambda,\Lambda,\alpha$, such that if
$$
|Du(0)|+|\sigma-2|+\|f\|_{L^\infty(B_1)}\le \kappa
$$
and
$$
|u(x)|\le \mu(1+|x|^{1+\overline{\alpha}})
$$
for $x\in\mathbb{R}^N$, then one has
$$
\sup_{x\in B_\rho}|u(x)|\le \rho^{1+\alpha}
$$
for every $\alpha\in(0,\overline{\alpha})$ with $\overline{\alpha}$ coming from Lemma \ref{lem3-4}.
\end{lemma}

\begin{proof}
Let $\varepsilon>0$ be a number to be chosen later. Thanks to Lemma \ref{lem3-4}, we find a $\kappa>0$ and an $F$-harmonic function $h$ such that $0\in\mathcal{C}(h)$ and
$$
\sup_{x\in B_{3/4}}|u(x)-h(x)|\le \varepsilon.
$$
In view of the optimal $C^{1,\overline{\alpha}}_{\rm loc}$ regularity for $h$ (see e.g., \cite{Caf89,Tex14}) together with the fact $h(0)=|Dh(0)|=0$, we derive
$$
\sup_{x\in B_r}|h(x)|\le Cr^{1+\overline{\alpha}} \quad\text{for any } r\in \Big(0,\frac{3}{4}\Big]
$$
with a universal constant $C=C(N,\lambda,\Lambda)>0$.

Next, choose two universal constants
$$
0<\rho\le\left(\frac{1}{2C}\right)^\frac{1}{\overline{\alpha}-\alpha} \quad\text{and}\quad \varepsilon=\frac{1}{2}\rho^{1+\alpha}.
$$
From this, it is easy to have
\begin{align*}
\sup_{x\in B_\rho}|u(x)|&\le\sup_{x\in B_\rho}|u(x)-h(x)|+\sup_{x\in B_\rho}|h(x)|\\
&\le \frac{1}{2}\rho^{1+\alpha}+C\rho^{\overline{\alpha}-\alpha}\rho^{1+\alpha}\\
&=\rho^{1+\alpha}.
\end{align*}
\end{proof}

\begin{lemma}
\label{lem3-6}
Let the preconditions of Lemma \ref{lem3-4} and $|f(x)|\le \mathfrak{m}|x|^\theta$ be true. Then for given $\mu>0$, one can find two small universal constants $\kappa,\rho\in(0,1)$ such that whenever
\begin{equation}
\label{3-6-1}
|\sigma-2|\le \kappa,        
\end{equation}
$$
|Du(0)|\le\kappa t^\alpha  \quad \text{for } t\in(0,\rho]
$$
and
$$
|u(x)|\le \mu(1+|x|^{1+\overline{\alpha}})
$$
for $x\in\mathbb{R}^N$ with $\overline{\alpha}$ from Lemma \ref{lem3-4}, it holds that
$$
\sup_{x\in B_t}|u(x)|\le Ct^{1+\alpha}
$$
with a universal number $C\equiv C(N,\lambda,\Lambda,\alpha)$ and
$$
\alpha\in(0,\overline{\alpha})\cap \left(0,\frac{\sigma-1+\theta}{1+p}\right].
$$
\end{lemma}

\begin{proof}
We verify this claim by induction argument. 
Observe that we may suppose that $\|f\|_{L^\infty(B_1)}\leq\kappa$ for $\kappa$ being from Lemma \ref{lem3-4} by scaling and normalization. First of all, we need to justify that for $k\in\mathbb{N}$
\begin{equation}
\label{3-6-2}
\sup_{x\in B_{\rho^k}}|u(x)|\le \rho^{k(1+\alpha)}
\end{equation}
under the conditions that \eqref{3-6-1} and $|Du(0)|\le \kappa\rho^{k\alpha}$ with $\rho$ from Lemma \ref{lem3-5}. For $k=1$, \eqref{3-6-2} follows immediately from Lemma \ref{lem3-5}. Suppose the conclusion \eqref{3-6-2} holds true for $k=1,2,\cdots,n$. Now we prove \eqref{3-6-2} for $k=n+1$. Set $u_n:\mathbb{R}^N\rightarrow\mathbb{R}$ as
$$
u_n(x)=\frac{u(\rho^n x)}{\rho^{n(1+\alpha)}}.
$$
Then $u_n$ is a viscosity solution to the following equation
$$
|Du_n|^p\overline{\mathcal{I}}(u_n,x)=\overline{f}(x), 
$$
where the nonlocal operator $\overline{\mathcal{I}}$ carries the same uniform ellipticity properties as the operator $\mathcal{I}$ in \eqref{main}, and
$$
\overline{f}(x)=\rho^{n\sigma-n(1+\alpha)-np\alpha}f(\rho^n x).
$$
Remember $\alpha\le \frac{\sigma-1+\theta}{1+p}$ and $|f(x)|\le \mathfrak{m}|x|^\theta$. There holds that
$$
|\overline{f}(x)|\le \mathfrak{m}\rho^{n(\sigma-1+\theta-\alpha-p\alpha)}|x|^\theta\le \mathfrak{m}.
$$

On the other hand, via the hypotheses of induction, we know $|Du(0)|\le \kappa\rho^{n\alpha}$, and so we get $|Du_n(0)|\le \kappa$.
In addition, we can demonstrate
$$
|u_n(x)|\le 1+|x|^{1+\overline{\alpha}} \quad\text{for } x\in\mathbb{R}^N.
$$
The evaluation of this inequality is analogous to that in Lemma \ref{lem3-3}. Thus we drop it here. At this point, $u_n$ falls into the framework of Lemma \ref{lem3-5}, and hence it yields
$$
\sup_{x\in B_{\rho}}|u_n(x)|\le \rho^{1+\alpha}.
$$
Rescaling back, one gets
$$
\sup_{x\in B_{\rho^{n+1}}}|u(x)|\le \rho^{(n+1)(1+\alpha)}.
$$
By now, the proof of \eqref{3-6-2} is finished. Finally, for $t\in(0,\rho]$, there exists an integer $n>0$ such that $\rho^{n+1}<t\le\rho^n$, which leads to
$$
\sup_{x\in B_t}|u_n(x)|\le \sup_{x\in B_{\rho^n}}|u(x)|\le \rho^{n(1+\alpha)}\le \rho^{-(1+\alpha)}t^{1+\alpha},
$$
provided $|Du(0)|\le\kappa t^\alpha$.
Now we complete the whole proof.
\end{proof}

Now we end this section with demonstrating Theorem \ref{thm2}.

\medskip

\noindent\textbf{Proof of Theorem \ref{thm2}.}
First notice that we can assume that $u(0)=0,\|u\|_{L^\infty(B_1)}\le1$ and $\|f\|_{L^\infty(B_1)}\le\kappa$ with $\kappa$ small enough, by translation and normalization. We are going to prove this conclusion by implementing a dichotomy argument, which is divided into two scenarios. Let $0<\kappa,\rho<1$ be two universal constants from Lemma \ref{lem3-6}.

\medskip

\textbf{Case 1}. $|Du(0)|\le \kappa\rho^\alpha$. Set a number
\begin{equation}
\label{t4-1}
\tau=\left(\frac{|Du(0)|}{\kappa}\right)^\frac{1}{\alpha}.
\end{equation}
In the first subcase $\tau\le t\le\rho$, it follows from \eqref{t4-1} that
$$
|Du(0)|\le \kappa t^\alpha.
$$
Thereby, exploiting Lemma \ref{lem3-6} leads to
$$
\sup_{x\in B_t}|u(x)|\le Ct^{1+\alpha}
$$
with a universal $C>0$. Furthermore, there holds that
\begin{align*}
\sup_{x\in B_t}|u(x)-Du(0)\cdot x|&\le Ct^{1+\alpha}+|Du(0)|t\\
&\le(C+\kappa)t^{1+\alpha}.
\end{align*}
Additionally, in the second subcase $0<t<\tau(\le \rho)$, we construct an auxiliary function
$$
u_\tau(x):=\frac{u(\tau x)}{\tau^{1+\alpha}}.
$$
It is easy to see that
\begin{equation}
\label{t4-1-1}
\sup_{x\in B_1}|u_\tau(x)|=\sup_{x\in B_1}\left|\frac{u(\tau x)}{\tau^{1+\alpha}}\right|\le C.
\end{equation}
Indeed, we have $|Du(0)|=\kappa\tau^\alpha$ by \eqref{t4-1} and employ again Lemma \ref{lem3-6} to obtain
$$
\sup_{x\in B_\tau}|u(x)|\le C\tau^{1+\alpha}.
 $$
Therefore, we can readily examine that $u_\tau\in C(\overline{B}_1)$ is a viscosity solution with \eqref{t4-1-1} to the problem
\begin{equation}
\label{t4-2}
|Du_\tau|^p\overline{\mathcal{I}}(u_\tau,x)=\overline{f}(x)  \quad \text{in } B_1,
\end{equation}
where the nonlocal operator $\overline{\mathcal{I}}$ has the same uniform ellipticity constants as the operator $\mathcal{I}$ in \eqref{main}, and $\overline{f}(x)=\tau^{\sigma-(1+\alpha)-p\alpha}f(\tau x)$. Note that $|\overline{f}(x)|\le \mathfrak{m}\tau^{\sigma-1+\theta-\alpha(1+p)}|x|^\theta\le\mathfrak{m}$ due to $\alpha\le\frac{\sigma-1+\theta}{1+p}$. Since the source term $\overline{f}$ in \eqref{t4-2} has a universal bound, the interior $C^{1,\alpha}$-regularity for $u_\tau$ follows from Theorem \ref{thm1}. In addition, recall the fact $|Du_\tau(0)|=\left|\frac{Du(0)}{\tau^\alpha}\right|=\kappa$. Then we may take a small universal radius $r>0$, independent of $\tau$, such that
$$
\frac{\kappa}{2}\le|Du_\tau(x)|\le 2\kappa  \quad\text{in } B_r.
$$
As a matter of fact, by $u_\tau\in C^{1,\alpha}_{\rm loc}(B_1)$, there holds that, for some universal $C\ge1$,
$$
|Du_\tau(x)-Du_\tau(0)|\le C|x|^\alpha
$$
and further
$$
\kappa-C|x|^\alpha\le |Du_\tau(x)|\le \kappa+C|x|^\alpha.
$$
From this, Eq. \eqref{t4-2} becomes
$$
\overline{\mathcal{I}}(u_\tau,x)=|Du_\tau|^{-p}\overline{f}(x)  \quad \text{in } B_r.
$$
That is to say, $u_\tau$ solves a uniformly elliptic nonlocal equation with a universally bounded source term on the right-hand side. As a consequence, through the regularity theory available in \cite{CS09, CS11}, such solution $u_\tau$ to the last problem is almost as regular as an $F$-harmonic function, as long as the parameter $\sigma$ is sufficiently close to 2. (This moment, the operator $\mathcal{I}_\sigma$ approaches a local uniformly $(\lambda,\Lambda)$-elliptic operator $F$.) That is, $u_\tau\in C^{1,\overline{\alpha}^-}_{\rm loc}$. We further derive
$$
\sup_{x\in B_\delta}|u_\tau(x)-Du_\tau(0)\cdot x|\le C\delta^{1+\alpha}
$$
for every $0<\delta\le \frac{r}{2}$ and $0<\alpha<\overline{\alpha}$. Scaling back, it yields that
$$
\sup_{x\in B_t}|u(x)-Du(0)\cdot x|\le Ct^{1+\alpha}
$$
for all $0<t\le\frac{\tau r}{2}$. When $t$ belongs to the interval $\Big(\frac{\tau r}{2},\tau\Big)$, we apply the first subcase with $t:=\tau$ to arrive at
\begin{align*}
\sup_{x\in B_t}|u(x)-Du(0)\cdot x|&\le\sup_{x\in B_\tau}|u(x)-Du(0)\cdot x|\\
&\le C\left(\frac{2}{r}\right)^{1+\alpha}\left(\frac{\tau r}{2}\right)^{1+\alpha}\\
&\le C\left(\frac{2}{r}\right)^{1+\alpha} t^{1+\alpha}=:Ct^{1+\alpha}.
\end{align*}
Here note $r$ is a universal constant. At this time, we infer $u $ is of class $C^{1,\alpha}$ at the origin in Case 1.

\medskip

\textbf{Case 2}. $|Du(0)|>\kappa \rho^\alpha$. In this scenario, consider a function
$$
v(x)=\frac{\kappa \rho^\alpha}{|Du(0)|}u(x),
$$
and then $|Dv(0)|=\kappa \rho^\alpha$, which is back to Case 1. The proof is complete now.    \hfill $\Box$


\section{$C^{1,\alpha}$-regularity for the case $\sigma\in(1,2)$}
\label{sec3}

In this part, we under the condition $\sigma\in(1,2)$ analyze the regularity theory for Eq. \eqref{main} via approximated viscosity solutions.

\begin{lemma}
\label{lem4-1}
Let the hypotheses $(A_1),(A_2)$ be in force. Suppose that $u\in C(\overline{B}_1)$ is an approximated viscosity solution of \eqref{2-1} with a continuous function $\Phi$. Then $u$ is a viscosity solution to \eqref{2-1} as well.
\end{lemma}

\begin{proof}
We only prove $u$ is a viscosity subsolution because the case of supersolution is analogous. Let $x_0\in B_1$ and $\varphi$ be a test function such that $u-\varphi$ attains its local maximum at $x_0$. Through Definition \ref{def2} there is a sequence $\{u_k\}$ satisfying $u_k\rightarrow u$ locally uniformly in $B_1$, so we have that there exists a sequence of smooth functions $\{\varphi_k\}$ such that $\varphi_k\rightarrow \varphi$ locally uniformly in $B_1$ and $u_k-\varphi_k$ reaches its local maximum at some point $x_k$ with $x_k\rightarrow x_0$. Owing to $u_k$ a viscosity solution of \eqref{2-2}, then
$$
-\Phi(x_k,|D\varphi_k(x_k)+\eta_k|+d_k)\mathcal{I}(u_k,\varphi_k,x_k)\le f(x_k),
$$
where $\eta_k\rightarrow\eta$ and $d_k\rightarrow0$. Remember that the functions $\Phi$ and $f$ are continuous, and that the convergence relation $\mathcal{I}(u_k,\varphi_k,x_k)\rightarrow\mathcal{I}(u,\varphi,x_0)$ follows from \cite[Lemma 5]{CS11}. Then we pass to the limit as $k\rightarrow\infty$ and derive
$$
-\Phi(x,|D\varphi(x_0)+\eta|)\mathcal{I}(u,\varphi,x_0)\le f(x_0).
$$
The proof is finished now.
\end{proof}

\begin{lemma}
\label{lem4-2}
Suppose that $\{u_k\}\subset C(\overline{B}_1)$ and $\{\eta_k\}\subset\mathbb{R}^N$ satisfy
$$
-\Phi_k(x,|Du_k+\eta_k|)\mathcal{I}(u_k,x)=f_k \quad \text{in } B_1
$$
in the approximated viscosity solution sense, where $f_k\in C(B_1)\cap L^\infty(B_1)$ and $\Phi_k$ fulfills $(A_3),(A_4)$ for all $k\in\mathbb{N}$. In addition, assume that there is such a function $u\in C(\overline{B}_1)\cap L^1_\sigma(\mathbb{R}^N)$ that $u_k\rightarrow u$ locally uniformly in $B_1$ and $\|f_k\|_{L^\infty(B_1)}\le \frac{1}{k}$. Then, $u$ is a viscosity solution to
$$
\mathcal{I}(u,x)=0   \quad \text{in } B_1.
$$
\end{lemma}

\begin{proof}
Via the definition of approximated viscosity solution, we can find sequences $\{u^i_k\}_i\subset C(B_1)\cap L^1_\sigma(\mathbb{R}^N)$, $\{\eta^i_k\}_i\subset\mathbb{R}^N$ and $\{d_i\}_i\subset\mathbb{R}^+$, satisfying that $u^i_k\rightarrow u_k$
locally uniformly in $B_1$ and $d_i\rightarrow0$ and $id^{s_2}_i\rightarrow\infty$ as $i\rightarrow\infty$, such that $u^i_k$ solves in viscosity sense
$$
-\Phi_k(x,|Du^i_k+\eta^i_k|+d_i)\mathcal{I}(u^i_k,x)=f_k \quad \text{in } B_1
$$
for each fixed $k$. Now consider the sequence $\{u^k_k\}$, the elements of which are the solutions to
$$
-\Phi_k(x,|Du^k_k+\eta^k_k|+d_k)\mathcal{I}(u^k_k,x)=f_k \quad \text{in } B_1
$$
for every $k$. Obviously, $u^k_k$ converges uniformly to $u$ locally in $B_1$. Besides, we can see
\begin{align*}
-\mathcal{I}(u^k_k,x)&\le\frac{\|f_k\|_{L^\infty(B_1)}}{\Phi_k(x,|Du^k_k+\eta^k_k|+d_k)}\\
&\le \frac{\overline{M}M\|f_k\|_{L^\infty(B_1)}}{\min\{(|Du^k_k+\eta^k_k|+d_k)^{s_1},(|Du^k_k+\eta^k_k|+d_k)^{s_2}\}}\\
&\le\frac{\overline{M}M}{kd_k^{s_2}},
\end{align*}
where we have utilized the properties of $\Phi_k$, and the relations $0\le s_1\le s_2$ and $0<d_k<1$ as $k$ is large enough. It follows from the fact $kd^{s_2}_k\rightarrow\infty$ ($k\rightarrow\infty$) in Definition \ref{def2} that $u$ is a viscosity subsolution of $-\mathcal{I}(u,x)\le0$ in $B_1$. The result that $u$ is a supersolution could be shown similarly.
\end{proof}

Next, consider the following equation
\begin{equation}
\label{4-1}
-\Phi(x,|Du_k|+d_k)\mathcal{I}(u_k,x)=f(x) \quad \text{in } B_1
\end{equation}
for any $k\in\mathbb{N}$. 
The forthcoming lemma states the viscosity solutions of \eqref{4-1} are locally Lipschitz continuous, which allows us to pass to the limit of the sequence of solutions $\{u_k\}$ by compactness argument.

\begin{lemma}
\label{lem4-3}
Let the hypotheses $(A_1)$--$(A_4)$ hold true. Suppose that $u_k\in C(\overline{B}_1)$ is a viscosity solution to \eqref{4-1}. Then $u_k$ has local Lipchitz continuity in $B_1$, that is,
$$
|u_k(x)-u_k(y)|\le C|x-y| \qquad \text{for } x,y\in B_{\frac{1}{2}}.
$$
Here the positive constant $C$ is independent of $d_k$.
\end{lemma}

The proof of this lemma is almost the same as that of Proposition \ref{pro3-1}, so we drop the details here. We now present an important result on the existence of approximated viscosity solutions to Eq. \eqref{main}.

\begin{proposition}
\label{pro4-4}
Assume that the preconditions $(A_1)$--$(A_4)$ are in force. Then we can find at least one approximated viscosity solution of \eqref{main}.
\end{proposition}

\begin{proof}
Introduce the approximation problem
\begin{equation*}
\begin{cases}-\Phi\left(x,|Du_k|+\left(\frac{1}{k}\right)^\frac{1}{2s_2}\right)\mathcal{I}(u_k,x)=f(x) \qquad \text{in } B_1, \\[2mm]
u_k(x)=g(x)  \qquad \text{on } \mathbb{R}^N\setminus B_1,
\end{cases}
\end{equation*}
where the boundary datum satisfies
$$
|g(x)|\le 1+|x|^{1+\alpha} \quad\text{for }  x\in\mathbb{R}^N
$$
with any $1+\alpha\in(0,\sigma)$. The existence of viscosity solution $u_k$ to the equation above follows from the paper \cite{BCI08}, because of the nonlocal operator $\mathcal{I}$ here carrying uniform ellipticity. Besides, recalling Lemma \ref{lem4-3}, we know that the viscosity solution $u_k$ is locally Lipschitz continuous in $B_1$ with the Lipschitz constant not depending on $k$. As a consequence, this guarantees that there is a function $u_\infty\in C^{0,\beta}_{\rm loc}(B_1)$ for some $\beta\in(0,1)$ such that $u_k$ converges locally uniformly to $u_\infty$ in $B_1$, and $u_\infty=g$ on $\mathbb{R}^N\setminus B_1$. Observe that by Definition \ref{def2} and via choosing $d_k=\left(\frac{1}{k}\right)^\frac{1}{2s_2}$ in \eqref{2-2} with $kd^{s_2}_k=k^\frac{1}{2}\rightarrow\infty$, we immediately see that $u_\infty$ is an approximated viscosity solution of Eq. \eqref{main}.
\end{proof}

In what follows, we are going to utilize the compactness argument to establish the approximation result below that is a crucial ingredient for showing Theorem \ref{thm3}.

\begin{lemma}
\label{lem4-5}
Suppose that the assumptions $(A_1)$--$(A_4)$ are satisfied. Let $u\in C(\overline{B}_1)$ be a normalized approximated viscosity solution of \eqref{2-1}. Then for given $\mu,\varepsilon>0$, there is a $\kappa>0$, depending on $N,\lambda,\Lambda,M,\overline{M},s_2,\sigma,\varepsilon,\mu$, such that if
$$
\|f\|_{L^\infty(B_1)}\le \kappa
$$
and
$$
|u(x)|\le \mu(1+|x|^{1+\hat{\alpha}})  \quad\text{for }  x\in\mathbb{R}^N
$$
with some $\hat{\alpha}\in(0,1)$, then there exists an $\mathcal{I}_\sigma$-harmonic function $h\in C^{1,\hat{\alpha}}_{\rm loc}(B_1)$ fulfilling
$$
\|u-h\|_{L^\infty(B_{3/4})}\le\varepsilon.
$$
\end{lemma}

\begin{proof}
Argue by contradiction. On the contrary, we suppose this statement fails. Then, there exist $\mu_0,\varepsilon_0>0$ for which we can find the sequences of functions $\{f_k\}$, $\{\Phi_k\}$ and $\{u_k\}\subset C(\overline{B}_1)\cap L^1_\sigma(\mathbb{R}^N)$, a sequence of vectors $\{\eta_k\}$, such that $u_k$ is an approximated viscosity solution of
\begin{equation}
\label{4-5-1}
-\Phi_k(x,|Du_k+\eta_k|)\mathcal{I}(u_k,x)=f_k(x) \quad \text{in } B_1,
\end{equation}
where
$$
\|f_k\|_{L^\infty(B_1)}\le \frac{1}{k}
$$
and
$$
\|u_k\|_{L^\infty(B_1)}\le1, \quad |u_k(x)|\le \mu_0(1+|x|^{1+\hat{\alpha}})  \quad\text{for }  x\in\mathbb{R}^N.
$$
Here the continuous function $\Phi_k(x,t)$ satisfies that $\overline{M}^{-1}\le \Phi_k(x,1)\le \overline{M}$ in $B_1$, and the map $t\rightarrow\frac{\Phi_k(x,t)}{t^{s_1}}$ is almost non-decreasing and the map $t\rightarrow\frac{\Phi_k(x,t)}{t^{s_2}}$ is almost non-increasing with the same constant $M\ge1$ in $(A_4)$. Nonetheless,
\begin{equation}
\label{4-5-2}
\|u_k-h\|_{L^\infty(B_{3/4})}>\varepsilon_0
\end{equation}
for any function $h\in C^{1,\hat{\alpha}}_{\rm loc}(B_1)$.

We have known from Lemma \ref{lem4-1} that the approximated viscosity solutions of \eqref{4-5-1} are viscosity solutions of such equation as well. In turn, making use of Lemma \ref{lem4-3}, we could find such a continuous function $u_\infty$ that $u_k$ converges uniformly to $u_\infty$ locally in $B_1$ by compactness. At this stage, it yields using the stability of approximated viscosity solutions (Lemma \ref{lem4-2}) and the preceding assumptions that $u_\infty$ is a viscosity solution of
$$
-\mathcal{I}_\sigma(u_\infty,x)=0 \quad \text{in }  B_\frac{7}{8}
$$
from \eqref{4-5-1}. Here it is well known that $u_\infty$ is of class $C^{1,\hat{\alpha}}_{\rm loc}(B_{7/8})$. Hence there is a contradiction with \eqref{4-5-2} by selecting $h:=u_\infty$ for $k$ large enough.
\end{proof}

\begin{lemma}
\label{lem4-6}
Let $u\in C(\overline{B}_1)$ be a normalized approximated viscosity solution of \eqref{main}. Under the assumptions of Lemma \ref{lem4-5}, for given $\mu>0$, one can find a $\kappa>0$, depending on $N,\lambda,\Lambda,M,\overline{M},s_2,\sigma,\alpha,\mu$, such that whenever
$$
\|f\|_{L^\infty(B_1)}\le \kappa
$$
and
$$
|u(x)|\le \mu(1+|x|^{1+\hat{\alpha}})  \quad\text{for }  x\in\mathbb{R}^N
$$
with $\hat{\alpha}$ from Lemma \ref{lem4-5}, then there is a constant $\rho\in\left(0,\frac{1}{2}\right)$ depending only on $N,\lambda,\Lambda,\alpha,\sigma$ and a sequence of affine functions $\{l_k\}$ with the form $l_k(x):=a_k+p_k\cdot x$ such that
$$
\sup_{x\in B_{\rho^k}}|u(x)-l_k(x)|\le\rho^{k(1+\alpha)},
$$
where
\begin{equation}
\label{4-6-1}
|a_{k+1}-a_k|+\rho^k|p_{k+1}-p_k|\le C\rho^{k(1+\alpha)}
\end{equation}
for any
$$
\alpha\in(0,\hat{\alpha})\cap\left(0,\frac{\sigma-1}{1+s_2}\right].
$$
Here the constant $C\ge1$ depends on $N,\lambda,\Lambda,\sigma$.
\end{lemma}

\begin{proof}
The proof of this lemma is analogous to that of Lemma \ref{lem3-3}, so we sketch it here.

\medskip

\textbf{Step 1}. First, we claim that there exists an affine map $l(x)$ and a number $0<\rho<\frac{1}{2}$ such that
$$
\sup_{x\in B_\rho}|u(x)-l(x)|\le \rho^{1+\alpha}.
$$
For $\varepsilon>0$ to be determined a posteriori, we suppose $h$ is an $\mathcal{I}_\sigma$-harmonic function that is $\varepsilon$-close to $u$ in $L^\infty(B_{3/4})$ (i.e., $\|u-h\|_{L^\infty(B_{3/4})}\le \varepsilon$). Owing to Lemma \ref{lem4-5}, the existence of such function $h$ could be assured, provided $\kappa>0$ is sufficiently small.

With the help of the $C^{1,\hat{\alpha}}$-regularity of $h$, there is an absolute constant $C\equiv C(N,\lambda,\Lambda,\sigma)$ satisfying
$$
\sup_{x\in B_\rho}|h(x)-(h(0)+Dh(0)\cdot x)|\le C\rho^{1+\hat{\alpha}}
$$
with
$$
|h(0)|+|Dh(0)|\le C.
$$
Set an affine function $l(x):=a+p\cdot x:=h(0)+Dh(0)\cdot x$. Via the triangular inequality,
\begin{align*}
\sup_{x\in B_\rho}|u(x)-l(x)|&\le\sup_{x\in B_\rho}|u(x)-h(x)|+\sup_{x\in B_\rho}|h(x)-l(x)|\\
&\le \varepsilon+C\rho^{1+\hat{\alpha}}\\
&\le\rho^{1+\alpha},
\end{align*}
where we have taken $\rho=\left(\frac{1}{2C}\right)^\frac{1}{\hat{\alpha}-\alpha}$ and $\varepsilon=\frac{1}{2}\rho^{1+\alpha}$. The universal choice of $\varepsilon$ determines the value of $\kappa$ by Lemma \ref{lem4-5}.

\medskip

\textbf{Step 2}. We proceed by induction. Suppose this assertion holds for $k=1,2,\cdots,n$. Next, we prove this in the case $k=n+1$. Define a function $u_n:\mathbb{R}^N\rightarrow \mathbb{R}$ as
$$
u_n(x):=\frac{u(\rho^nx)-l_n(\rho^nx)}{\rho^{n(1+\alpha)}}.
$$
Hereafter, the subsequent processes are the same as the proof of Step 2 in Lemma \ref{lem3-3} virtually, so we omit the details here.
\end{proof}

\noindent\textbf{Proof of Theorem \ref{thm3}}. Since the sequences $\{a_k\}, \{p_k\}$ are Cauchy sequences from \eqref{4-6-1}, then we may pick $a_\infty\in\mathbb{R}$, $p_\infty\in\mathbb{R}^N$ such that
$$
|a_k-a_\infty|+|p_k-p_\infty|\rightarrow0 \qquad (k\rightarrow\infty),
$$
and further
$$
|a_k-a_\infty|+\rho^k|p_k-p_\infty|\le C\rho^{k(1+\alpha)}.
$$
For any $r\in\Big(0,\frac{1}{2}\Big]$, select $k\in\mathbb{N}$ such that $\rho^{k+1}<r\le \rho^k$ with $\rho$ a fixed sufficiently small number. Then, let $l_\infty(x)=a_\infty+p_\infty\cdot x$ and we evaluate
\begin{align*}
\sup_{x\in B_r}|u(x)-l_\infty(x)|&\le\sup_{x\in B_{\rho^k}}|u(x)-l_k(x)|+\sup_{x\in B_{\rho^k}}|l_k(x)-l_\infty(x)|\\
&\le C\rho^{-(1+\alpha)}\rho^{(k+1)(1+\alpha)}\\
&\le Cr^{1+\alpha},
\end{align*}
which indicates $u$ is of class $C^{1,\alpha}$; see for instance \cite{BPRT20}. So far, we have demonstrated the approximated viscosity solutions are locally $C^{1,\alpha}$-regular. On the other hand, employing Proposition \ref{pro4-4}, the existence of such solutions can be guaranteed. Moreover, it follows from Lemma \ref{lem4-1} that approximated viscosity solutions are viscosity solutions as well. As a result, we verify the existence of viscosity solution with $C^{1,\alpha}_{\rm loc}(B_1)$-regularity to \eqref{main}.

\section{H\"{o}lder continuity of solutions}
\label{sec4}

We in this section provide the proof of local H\"{o}lder regularity for viscosity solutions to \eqref{main} by using  Ishii-Lions argument again and elliptic estimates in the direction of gradient shown in \cite{BCCI12, BCI11}.

\medskip

\noindent\textbf{Proof of Theorem \ref{thm4}.}
(1) The proof of this theorem is analogous to that of Proposition \ref{pro3-1}, so we profile it and follow the notations there. Especially, we stress that for $t\in[0,+\infty)$
$$
\omega(t)=Lt^\sigma
$$
with $\sigma\in(0,1)$ at this time. For the function with ``double variables"
$$
\Psi(x,y)=u(x)-u(y)-\varphi(x,y),
$$
it attains maximum in $\overline{B}_1\times\overline{B}_1$ at $(\overline{x},\overline{y})$ by continuity. In the sequel, we focus on demonstrating $\Psi(\overline{x},\overline{y})\le0$ via arguing by contradiction. Suppose $\Psi(\overline{x},\overline{y})>0$. Then we can obtain the following information:
$$
\overline{x}\neq\overline{y}, \ \ \overline{y}\in B_\frac{3}{4} \ \ \text{and} \ \ |\overline{x}-\overline{y}|\le\left(L^{-1}\mathop{\rm{osc}}\limits_{B_1}u\right)^{\sigma^{-1}}
$$
and further
$$
|\overline{x}|\le \frac{3}{4}+\left(L^{-1}\mathop{\rm{osc}}\limits_{B_1}u\right)^{\sigma^{-1}}\le \frac{7}{8}
$$
via selecting $L\ge 8^\sigma\mathop{\rm{osc}}\limits_{B_1}u$.

Through applying the fact that $\Psi$ has its maximum at $(\overline{x},\overline{y})\in\overline{B}_1\times\overline{B}_1$ and the definition of the solution, we arrive at the viscosity inequalities
$$
-\Phi(\overline{x},|D_x\varphi(\overline{x},\overline{y})|)\mathcal{I}^\delta(u,\varphi(\cdot,\overline{y}),\overline{x})\le f(\overline{x})
$$
and
$$
-\Phi(\overline{y},|-D_y\varphi(\overline{x},\overline{y})|)\mathcal{I}^\delta(u,-\varphi(\overline{x},\cdot),\overline{y})\ge f(\overline{y})
$$
for each $\delta\in(0,1)$. Compute the partial derivative of $\varphi$,
\begin{equation*}
 \begin{cases}
D_x\varphi(\overline{x},\overline{y})=\sigma L|\overline{x}-\overline{y}|^{\sigma-2}(\overline{x}-\overline{y}), \\[2mm]
D_y\varphi(\overline{x},\overline{y})=-\sigma L|\overline{x}-\overline{y}|^{\sigma-2}(\overline{x}-\overline{y})+Dh(\overline{y}).
\end{cases}
\end{equation*}
Furthermore, we can find that, by taking such large $L\ge1$ that $C(N)\Big(\mathop{\rm{osc}}\limits_{B_1}u+1\Big)\le2^{\sigma-2}\sigma L$,
\begin{equation*}
\begin{cases}
|D_x\varphi(\overline{x},\overline{y})|=\sigma L|\overline{x}-\overline{y}|^{\sigma-1}\ge2^{\sigma-1}\sigma L \\[2mm]
|D_y\varphi(\overline{x},\overline{y})|\ge \sigma L|\overline{x}-\overline{y}|^{\sigma-1}-\Big(\mathop{\rm{osc}}\limits_{B_1}u+1\Big)|D\psi(\overline{y})|\ge 2^{\sigma-1}\sigma L-C(N)\Big(\mathop{\rm{osc}}\limits_{B_1}u+1\Big)\ge2^{\sigma-2}\sigma L.
\end{cases}
\end{equation*}
Based on the structural conditions on $\Phi$ and the lower bound on $|D_x\varphi(\overline{x},\overline{y})|$ or $|D_y\varphi(\overline{x},\overline{y})|$, we can infer
$$
-\mathcal{I}^\delta(u,\varphi(\cdot,\overline{y}),\overline{x})\le\frac{\|f\|_{L^\infty(B_1)}}{\Phi(\overline{x},|D_x\varphi(\overline{x},\overline{y})|)}
\le \frac{C\|f\|_{L^\infty(B_1)}}{(2^{\sigma-2}\sigma L)^{s_1}}
$$
and
$$
-\mathcal{I}^\delta(u,-\varphi(\overline{x},\cdot),\overline{y})\ge-\frac{\|f\|_{L^\infty(B_1)}}{\Phi(\overline{y},|D_y\varphi(\overline{x},\overline{y})|)}
\ge-\frac{C\|f\|_{L^\infty(B_1)}}{(2^{\sigma-2}\sigma L)^{s_1}},
$$
that is,
$$
\mathcal{I}^\delta(u,\varphi(\cdot,\overline{y}),\overline{x})-\mathcal{I}^\delta(u,-\varphi(\overline{x},\cdot),\overline{y})\ge-\frac{C\|f\|_{L^\infty(B_1)}}{L^{s_1}},
$$
where the constant $C>0$ depends on $\sigma, s_1,N,M,\overline{M}$.

When evaluating the nonlocal term on the left-hand side of the last display, we can use the estimates of $T_2$ in Step 3 of Theorem 1-(i) in \cite{BCI11}. Then we choose $L$ sufficiently large to obtain a contradiction. Then the H\"{o}lder continuity is proved, and the upper bound on H\"{o}lder norm is indeed the $L$. Here let us mention that the subsequent procedure is the same as that in Proposition \ref{pro3-1}.

\smallskip

(2) For the case $\sigma=1$, the proof is very similar to that in the scenario $\sigma\in(0,1)$. It is worth stressing that we need to introduce the auxiliary function $\omega(t)=Lt^\alpha$ with any $\alpha\in(0,1)$.   \hfill $\Box$


\section*{Acknowledgements}
This work was supported by the National Natural Science Foundation of China (Nos. 12071098, 11871134) and the Young talents sponsorship program of Heilongjiang Province (No. 2023QNTJ004). 

\section*{Declarations}
\subsection*{Conflict of interest} The authors declare that there is no conflict of interest. We also declare that this
manuscript has no associated data.

\subsection*{Data availability} Data sharing is not applicable to this article as no datasets were generated or analysed
during the current study.

\end{document}